\documentclass[11pt,reqno, twoside]{amsart}
\usepackage{amscd}
\usepackage{amsfonts}
\usepackage{amsmath}
\usepackage{amssymb}
\usepackage{amsthm}
\usepackage{fancyhdr}
\usepackage{latexsym}
\usepackage[colorlinks=true, pdfstartview=FitV, linkcolor=blue, citecolor=blue, urlcolor=blue]{hyperref}
\usepackage{enumitem}      
\usepackage{mathtools}
\usepackage{mathabx}
\usepackage{upgreek}
\usepackage{indentfirst} % for indentation after chapter, section, subsection etc.        
\usepackage{schemata} %for brackets around paragraphs
\usepackage{blindtext, rotating}   % for rotating text or images
\usepackage{soul} % for striking through text
\setstcolor{red}
\usepackage{color}
\usepackage{tikz}
\usetikzlibrary{arrows.meta}
\usetikzlibrary{decorations.markings}
\tikzset{->-/.style={decoration={
  markings,
  mark=at position #1 with {\arrow{>}}},postaction={decorate}}}
  \tikzset{middlearrow/.style={
        decoration={markings,
            mark= at position 0.55 with {\arrow{#1}} ,
        },
        postaction={decorate}
    }
}
\usepackage{caption}        
\newcommand{\eee}[1]{\begin{equation}#1\end{equation}}
\newcommand{\sss}[1]{\begin{subequations}#1\end{subequations}}

\newcommand{\ddd}[1]{\begin{alignat}{2}#1\end{alignat}}

\newcommand{\nn}{\nonumber}

\definecolor{ddgreen}{RGB}{0,170,0}

\renewcommand{\b}{\mathcolor{blue}}

\newcommand{\no}[1]{\left\| #1 \right\|}

\makeatletter
\def\mathcolor#1#{\@mathcolor{#1}}
\def\@mathcolor#1#2#3{%
\protect\leavevmode
\begingroup
\color#1{#2}#3%
\endgroup
}
\makeatother
\theoremstyle{plain}  % default
\newtheorem{theorem}{Theorem}[section]
\newtheorem{proposition}{Proposition}[section]
\newtheorem{lemma}{Lemma}[section]
\newtheorem{corollary}{Corollary}[section]

\theoremstyle{definition}

\newtheorem{remark}{Remark}[section]
\newenvironment{Proof}[1][\proofname]
{\proof[\textnormal{\textbf{#1.}}]}{\endproof}
\newcommand{\bp}{\begin{Proof}}
\newcommand{\ep}{\end{Proof}}

\DeclareMathSizes{12}{12}{7}{5}
\numberwithin{figure}{section}
\numberwithin{equation}{section}
\usepackage{geometry}
\makeatletter
\def\l@section{\@tocline{1}{0pt}{1pc}{}{}}
\def\l@subsection{\@tocline{2}{0pt}{1pc}{4.6em}{}}
\def\l@subsubsection{\@tocline{3}{0pt}{1pc}{7.6em}{}}
\renewcommand{\tocsection}[3]{%
  \indentlabel{\@ifnotempty{#2}{\makebox[2.3em][l]{%
    \ignorespaces#1 #2.\hfill}}}#3}
\renewcommand{\tocsubsection}[3]{%
  \indentlabel{\@ifnotempty{#2}{\hspace*{2.3em}\makebox[2.3em][l]{%
    \ignorespaces#1 #2.\hfill}}}#3}
\renewcommand{\tocsubsubsection}[3]{%
  \indentlabel{\@ifnotempty{#2}{\hspace*{4.6em}\makebox[3em][l]{%
    \ignorespaces#1 #2.\hfill}}}#3}
\makeatother 
\makeatletter
\@namedef{subjclassname@2020}{%
  \textup{2020} Mathematics Subject Classification}
\makeatother
\begin{document}
	\thanks{Bradley Isom is partially supported by a graduate research assistantship  from grant NSF-DMS 1908626. Atanas Stefanov acknowledges partial support from grant NSF-DMS 1908626. }
\title{Growth bound and nonlinear smoothing for the Periodic Derivative Nonlinear Schr\"odinger Equation}
\author{Bradley Isom$^*$, Dionyssios Mantzavinos \& Atanas Stefanov
\vskip 3mm
{\tiny Department of Mathematics, University of Kansas, Lawrence, KS 66045}}
\begin{abstract}
A polynomial-in-time growth bound is established for global Sobolev $H^s(\mathbb T)$ solutions to the derivative nonlinear Schr\"odinger equation on the circle with $s>1$. These  bounds are derived as a consequence of a nonlinear smoothing effect for an appropriate gauge-transformed version of the periodic Cauchy problem,  according to which a solution with its linear part  removed possesses higher spatial regularity than the initial datum associated with that solution. 
\end{abstract}
\date{December 17, 2020. $^*$\!\textit{Corresponding author}: bsisom@ku.edu. On behalf of all authors, the corresponding author states that there is no conflict of interest. The manuscript has no associated data.}

\maketitle
\markboth
{Growth Bound and Nonlinear Smoothing for the Periodic Derivative Nonlinear Schr\"odinger Equation}
{B. Isom, D. Mantzavinos \& A. Stefanov}
%
%
%\tableofcontents
%

\section{Introduction and Results}

We consider the Cauchy problem for the derivative nonlinear Schr\"{o}dinger (dNLS) equation on the circle
\begin{subequations}\label{cauchy-dnls}
\ddd{
&u_t-iu_{xx}=\partial_x(|u|^2u),~~x\in\mathbb{T}, \  t\in\mathbb{R}, \label{dnls-eq}\\
&u(x,0)=u_0(x)\in H^s(\mathbb{T})\label{dnls-id},
}
\end{subequations}
where $u=u(x, t)$ is a complex-valued function, $\mathbb{T}=\mathbb R/2\pi \mathbb Z$ is the one-dimensional torus (circle), and $H^s(\mathbb{T})$ is the $L^2$-based Sobolev space on the circle. 

The dNLS equation was derived as a model in plasma physics in the 1970s, see \cite{mmot1976}. As shown in \cite{kn1978}, it is a completely integrable system, possessing a Lax pair formulation and an infinite number of conserved quantities, including the following:
\ddd{
M(u)
&=
\int_{\mathbb{T}} |u|^2~dx,
\quad
P(u)
=
\int_{\mathbb{T}} \left[\text{Im}(u\overline{u_x})+\tfrac{1}{2}|u|^4\right]dx,
\label{mass}
\\
E(u)
&=
\int_{\mathbb{T}} \left[|u_x|^2 +\tfrac{3}{2} |u|^2 \text{Im}(u\overline{u_x})+\tfrac{1}{2}|u|^6\right]dx,
\nn
}
where $M(u)$, $P(u)$ and $E(u)$ correspond to the mass, momentum and energy, respectively, of the solution. Note that $P(u)$ is the Hamiltonian for \eqref{cauchy-dnls}.

Concerning the well-posedness of \eqref{cauchy-dnls}, Fukuda and Tsutsumi \cite{ft1980} showed local well-posedness in $H^s$, $s>3/2$, on both the line and the circle using the method of parabolic regularization. Furthermore, in \cite{ft1981} they demonstrated global well-posedness of solutions in $H^2$ with sufficiently small norm $\no{u_0}_{H^1}$. Hayashi and Ozawa \cite{h1993, ho1992, ho1994} improved upon this result in the Euclidean setting by showing global well-posedness of solutions in $H^1(\mathbb{R})$ with sufficiently small norm $\no{u_0}_{L^2(\mathbb{R})}$. In particular, their result was obtained by first performing a gauge transformation of \eqref{dnls-eq}, which removed the term $|u|^2 u_x$ from the nonlinearity. Takaoka \cite{t1999} combined the gauge transformation of Hayashi and Ozawa and the Fourier restriction norm method introduced by Bourgain in the breakthrough paper \cite{b1993} to establish local well-posedness in $H^{1/2}(\mathbb{R})$. This result was shown to be sharp  by Biagoni and Linares \cite{bl2001} in the sense that the data-to-solution map fails to be uniformly continuous for $s<1/2$. Thus, $s=1/2$ is the optimal result attainable for the well-posedness of \eqref{cauchy-dnls} using a fixed point argument on the gauge equation, although the critical regularity for scaling in the Euclidean setting is at the level of $s=0$. Under the assumption of a sufficiently small $\no{u_0}_{L^2(\mathbb{R})}$ norm, Colliander, Keel, Staffilani, Takaoka and Tao \cite{ckstt2002} obtained global well-posedness for $s>1/2$. Finally, global well-posedness for $s=1/2$ was demonstrated by Miao, Wu and Xu  \cite{mwx2011} and later by Guo and Wu  \cite{gw2017}, with the latter work improving the restriction on the initial data from $\no{u_0}_{L^2(\mathbb{R})}<\sqrt{2\pi}$ to $\no{u_0}_{L^2(\mathbb{R})}<\sqrt{4\pi}$. Such mass restrictions come from the sharp Gagliardo-Nirenberg inequalities. Finally, it is worth mentioning that  Jenkins, Liu, Perry and Sulem \cite{jlps2019} and, more recently, Bahouri and Perelman \cite{bp2020} demonstrate global well-posedness of dNLS with initial data $u_0$ in the weighted Sobolev space $H^{2,2}(\mathbb{R})$ and in $H^{1/2}(\mathbb{R})$, respectively,  without a mass restriction.

The majority of the above results concern the Cauchy problem on the line. Regarding the periodic problem \eqref{cauchy-dnls}, local well-posedness in $H^{1/2}(\mathbb{T})$ was established by Herr \cite{h2006} by adapting the gauge transformation of Hayashi and Ozawa to the periodic setting. The same article gives global well-posedness for $u_0\in H^1(\mathbb{T})$ such that $\no{u_0}_{L^2(\mathbb{T})}<2/3$. This mass threshold was improved by Mosincat and Oh in \cite{mo2015}, where they show global well-posedness in $H^1(\mathbb{T})$ for $\no{u_0}_{L^2(\mathbb{T})}<\sqrt{4\pi}$. Using the $I$-method, Win \cite{w2010} obtained global well-posedness in $H^s(\mathbb{T})$ for $s>1/2$ under the  assumption of a sufficiently small $\no{u_0}_{L^2(\mathbb{T})}$ norm. Finally, Mosincat \cite{m2017} established global well-posedness in $H^{1/2}(\mathbb{T})$ provided that $\no{u_0}_{L^2(\mathbb{T})}<\sqrt{4\pi}$. At the time of writing of the present article, the mass restriction in the periodic case had not been removed.

\vskip 3mm
\noindent
\textbf{Notation.} 
In order to state the main results of this work, we  introduce the following notation. 
\vskip 1mm
\begin{enumerate}[label=$\bullet$, leftmargin=4mm, rightmargin=0mm]
\advance\itemsep 2mm
\item
For $a,b>0$, we write $a\lesssim b$ if there exists $C>0$ such that $a\leqslant C b$. If $a\lesssim b$ and $b\lesssim a$ then we write $a\sim b$. Furthermore, if $C\geq 10^6$ and $a<\frac{1}{C}\, b$ with $a\not\sim b$ then we write $a\ll b$.
\item For $f\in L^p(\mathbb{T})$, $1\leq p\leq \infty$, we define the spatial Fourier transfom of $f$, denoted by $\mathcal{F}_x(f)=\widehat{f}$, as
\eee{\label{ft-s}
\mathcal{F}_x(f)(\xi) = \widehat{f}(\xi)
:=
\frac{1}{\sqrt{2\pi}}
\int_{\mathbb{T}} e^{-i\xi x} f(x)~dx,
~~\xi\in\mathbb{Z}.
}
Furthermore, for $f\in L^2(\mathbb{T})$, we have the inversion formula
\eee{\label{inversion}
f(x)
=
\frac{1}{\sqrt{2\pi}}
\sum_{\xi\in\mathbb{Z}} e^{i\xi x} \widehat{f}(\xi).
}
For $f\in \mathcal{S}(\mathbb{R})$, the space of Schwartz functions, we define the temporal Fourier transform of $f$, denoted by $\mathcal{F}_t(f)$, as 
\eee{\label{ft-t}
\mathcal{F}_t(f)(\tau)
:=
\frac{1}{\sqrt{2\pi}}
\int_{\mathbb{R}}
e^{-it\tau} f(t)~dt,
~~\tau\in\mathbb{R}.
}
Finally, for $f\in \mathcal{S}(\mathbb{R}_t; L^p(\mathbb{T}_x))$ we denote the spatiotemporal Fourier transform of $f$ by 
\ddd{\label{ft-st}
\mathcal{F}_t\mathcal{F}_x(f)(\xi,\tau)
=
\widetilde{f}(\xi,\tau).
}
\item 
We define the Bessel potential $J_x^s$ via Fourier transform as
\sss{
\eee{
\widehat{J_x^s f}(\xi)
:=
\langle{\xi\rangle}^s\widehat{f}(\xi), \quad \langle{\cdot\rangle} := \left(1+|\cdot|^2\right)^{\frac 12}.
}
Then, for any $s\geqslant 0$ and $p\geqslant 1$, we define the Bessel potential space
\eee{
H^{s,p}(\mathbb T) := \left\{f\in L^p(\mathbb T) : \no{f}_{H^{s,p}(\mathbb T)} := \no{J_x^s f}_{L^p(\mathbb T)} <\infty \right\},
}
}
In the special case $p=2$, the above space reduces to the 
  Sobolev space $H^s(\mathbb T)$.
\item
For any $s, b  \in \mathbb R$, we define the Bourgain space $X^{s, b}$ as the closure of $\mathcal{S}(\mathbb{R}_t;C^{\infty}(\mathbb{T}_x))$ under the norm
\eee{
\no{f}_{X^{s, b}} := \no{\left\langle \xi \right\rangle^s \left\langle \tau+\xi^2\right\rangle^b \widetilde f(\xi, \tau)}_{\ell^2_{\xi}L^2_{\tau}}.
}
Similarly, the space $Y^{s,b}$ is defined via the norm
\eee{
\no{f}_{Y^{s,b}}
:=
\no{\langle{\xi\rangle}^s \langle{\tau+\xi^2\rangle}^b \widetilde{f}(\xi,\tau)}_{\ell^2_{\xi} L^1_{\tau}}.
}
In addition, we define the Banach space $Z^s:= X^{s,\frac{1}{2}}\cap Y^{s,0}$ with  norm
\eee{
\no{f}_{Z^s}
:=
\no{f}_{X^{s,\frac{1}{2}}}+\no{f}_{Y^{s,0}}.
}
Finally, the restriction of $Z^s$ on $\mathbb T \times [0, T]$ with $T>0$ is denoted by $Z^s_T$ and is defined via the norm
\eee{
\no{f}_{Z^s_T}
:=
\inf\left\{ \no{g}_{Z^s} : g|_{[0,T]} = f\right\}.
}
\item
We define the Littlewood-Paley-type projection operator $P_k$ by 
\eee{
\widehat{P_k(f)}(\xi)
:=
\begin{cases}
\chi_{\{\xi=0\}}\widehat{f}(0),~k=0 \\
\chi_{\left\{2^{k-1} \leqslant |\xi|< 2^k\right\}}\widehat{f}(\xi),~k\in\mathbb{N},
\end{cases}
}
where $\chi_A$ is the characteristic function of the set $A$.  We will often denote $P_k(f)$ simply by $f_k$. By this definition, it follows that
\eee{
\sum_{k=0}^{\infty}\widehat{f_k}(\xi)
=
\widehat{f}(\xi), \quad \xi\in\mathbb Z.}
\item We let $\eta\in C_0^{\infty}(-2,2)$ with $0\leq \eta \leq 1$ and $\eta(t)=1$ for all $t\in [-1,1]$. For $T>0$, we define $\eta_T(t) :=\eta(t/T)$.
\item Following \cite{h2006}, we introduce the periodic gauge transformation of a solution $u$ to \eqref{cauchy-dnls} by
\ddd{\label{v-gauge}
v(x,t)
=
\mathcal{G}(u)(x,t)
:=
e^{-i\mathcal{I}(u)(x,t)}u(x,t),
}
where $\mathcal{I}(u)(x,t)$ is the mean-zero spatial primitive of $|u(x,t)|^2 - \frac{1}{2\pi}\no{u(t)}_{L^2(\mathbb{T})}^2$ given by
$$
\mathcal{I}(u)(x,t)
:=
\frac{1}{2\pi}\int_0^{2\pi}\int_{\theta}^x \left[ |u(y,t)|^2-\frac{1}{2\pi}\no{u(t)}_{L^2(\mathbb{T})}^2 \right] dyd\theta.
$$
Let
\eee{
\mu:=\frac{1}{2\pi}\no{u_0}_{L^2(\mathbb{T})}^2 = \frac{1}{2\pi}\no{u(t)}^2_{L^2(\mathbb{T})}, \quad t\in \mathbb R,
}
where the second equality is due to the conservation of mass in \eqref{mass}. In fact, we further have $\mu=\frac{1}{2\pi}\no{v(t)}^2_{L^2(\mathbb{T})}$. A straightforward computation then shows that  $v$ satisfies the equation
\eee{\label{v-gauge-pde}
v_t-iv_{xx}-2\mu v_x
=
-v^2 \overline{v}_x+\frac{i}{2}|v|^4 v-i\mu|v|^2 v+i\psi(v) v,
}
where 
\eee{\nn
\psi(v)(t):=\frac{1}{2\pi}\int_0^{2\pi} \left[2 \text{Im}(\overline{v}_x v)(\theta, t)-\frac{1}{2}|v|^4(\theta, t)\right]d\theta+\mu^2.
}
The term $2\mu v_x$ can be removed from \eqref{v-gauge-pde} by means of the transformation
\eee{\label{w-gauge}
w(x,t)
=
\tau_{-\mu}v(x,t)
:=
v(x-2\mu t,t).
}
Indeed, since $\tau_{-\mu}$ commutes with $\psi$ and is an isometry on $L^2(\mathbb{T})$, we find that $w$ satisfies
\ddd{
w_t-iw_{xx}
=
-w^2 \overline{w}_x+\frac{i}{2}|w|^4 w-i\mu|w|^2 w+i\psi(w) w.
\label{w-pde}}
Finally, we introduce a second gauge transformation,
\eee{\label{z}
z(x,t)
:=
e^{-ig(t)}w(x,t),
}
where
\eee{\nn
g(t):= \frac{8\pi^3-1}{16\pi^4} \int_0^t \no{w(t')}_{L^4(\mathbb{T})}^4 dt'-\mu^2 t.
}
We note that $z$ is related to $u$ as follows:
\eee{
z(x,t)
=
e^{-ig(t)}
e^{-i\mathcal{I}(u)(x-2\mu t,t)}
u(x-2\mu t,t)
=
e^{-ig(t)} \tau_{-\mu} \mathcal{G}(u)(x,t).
\nn
}
Also, it will be shown in Section \ref{s-dbp} that  $z$ satisfies the Cauchy problem \eqref{gauge-ivp} as well as the integral equation \eqref{duhamel-z}. 
\end{enumerate}

\vskip 3mm

With the above notation  in place, we now state some essential previous results and then introduce the main results of this work. We begin with the well-posedness of the gauge-equivalent Cauchy problem \eqref{gauge-ivp}, which follows from Theorem 5.1 of \cite{h2006}.
\begin{theorem}[\b{Well-posedness of the gauge equation -- \cite{h2006}, Theorem 5.1}]
\label{herr-wp}
Suppose $z_0 \in H^s(\mathbb{T})$ with $s\geq1/2$. Then, there exists a non-increasing function $T: [0,\infty)\rightarrow [0,\infty)$ with $T=T(\no{z_0}_{H^s(\mathbb{T})})$ and a unique $z\in Z_T^s$ satisfying the gauge-equivalent Cauchy problem \eqref{gauge-ivp} in the Duhamel sense with the estimate
\ddd{\label{wp-est}
\no{z}_{Z_T^s}
&\leq
c \no{z_0}_{H^s}.
}
Furthermore, the data-to-solution map is Lipschitz from bounded subsets of $H^s(\mathbb{T})$ to bounded subsets of $Z_T^s$.
\end{theorem}
\begin{remark}
In \cite{h2006}, it is stated that the local time of existence for the solution $z$ can be taken to depend only on $\no{z_0}_{H^{1/2}(\mathbb{T})}$ instead of $\no{z_0}_{H^s(\mathbb{T})}$, namely, $T=T(\no{z_0}_{H^{1/2}(\mathbb{T})})$.
\end{remark}
Next, we recall the well-posedness of the dNLS Cauchy problem \eqref{cauchy-dnls} as guaranteed by Theorem 1.1 of \cite{h2006}.
\begin{theorem}[\b{Well-posedness of  dNLS on $\mathbb T$ -- \cite{h2006}, Theorem 1.1}]
\label{herr-wp1}
Suppose $u_0\in H^s(\mathbb{T})$ with $s\geq 1/2$. If $z \in Z_T^s$ is the solution to the gauge-equivalent Cauchy problem \eqref{gauge-ivp} as guaranteed by Theorem \ref{herr-wp}, then $u=e^{ig(t)} \mathcal{G}^{-1}(\tau_{\mu} z)\in C([0,T];H^s(\mathbb{T}))$ is the unique solution satisfying the dNLS Cauchy problem \eqref{cauchy-dnls} in the sense of Duhamel. Furthermore, $u$ is a limit of smooth solutions.
\end{theorem}
\begin{remark}
Lemma \ref{u-z} implies that there exists a non-increasing function $\tilde{T} : [0,\infty)\rightarrow [0,\infty)$ such that $\tilde{T}=\tilde{T}(\no{u_0}_{H^s(\mathbb{T})})$ and $\tilde{T}(\no{u_0}_{H^s(\mathbb{T})})\leq T(\no{z_0}_{H^s(\mathbb{T})})$. Therefore, for $s\geq 1/2$ the time of existence in Theorem \ref{herr-wp1} may be taken to depend on $\no{u_0}_{H^{1/2}(\mathbb{T})}$ instead of $\no{u_0}_{H^s(\mathbb{T})}$.
\end{remark}
Finally, we state the following   global existence result from \cite{h2006}.
\begin{corollary}[\b{Global existence -- \cite{h2006}, Corollary 1.2}]\label{global}
For $s\geq 1$ let $u\in C([0,T]; H^s(\mathbb{T}))$ be the solution to the Cauchy problem \eqref{cauchy-dnls} from Theorem \ref{herr-wp1}. Then, for $\no{u_0}_{L^2(\mathbb{T})}$ sufficiently small, 
\ddd{\label{h1-est}
\no{u(t)}_{H^1(\mathbb{T})}
&\leq
C(\no{u_0}_{H^1(\mathbb{T})}), \quad t\in [0, T].
}
Consequently, the time of existence for the solution $u$ can be taken arbitrarily large.
\end{corollary}
\begin{remark}
The above global result follows from the observation that the local time of existence $T$ is bounded below by a function of $\no{u_0}_{H^1(\mathbb{T})}$. While global well-posedness in $H^{1/2}(\mathbb{T})$ has been obtained in \cite{m2017}, an estimate of the form \eqref{h1-est} does not seem to hold with $H^1(\mathbb{T})$ replaced by $H^{1/2}(\mathbb{T})$.
\end{remark}
\indent 
The main goal of this paper is to establish a polynomial-in-time bound on the growth of global solutions to the periodic dNLS Cauchy problem \eqref{cauchy-dnls}. Key to demonstrating this bound is the discovery of a local nonlinear smoothing effect for  the gauge problem \eqref{gauge-ivp}, according to which the solution $z$ of \eqref{gauge-ivp} with the linear part removed possesses higher spatial regularity than the initial data $z_0$. The effect is more readily seen by first recasting \eqref{gauge-ivp} into the Duhamel form \eqref{duhamel-z} via  a method known as differentiation by parts  (see Section \ref{s-dbp}). The precise statement of this first result is as follows.
\begin{theorem}[\b{Nonlinear smoothing}]
\label{dbp-wp}
Suppose $s>1/2+\varepsilon$ with $0<\varepsilon\ll 1/2$ and let $z_0\in H^s(\mathbb{T})$. Then, for $T=T(\no{z_0}_{H^{1/2+\varepsilon}(\mathbb{T})})$, the solution $z\in Z^s_T$  of  the Cauchy problem \eqref{gauge-ivp} from Theorem \ref{herr-wp}  satisfies the integral equation \eqref{duhamel-z}. 

Moreover,   for $0<a<\min\{s-1/2-\varepsilon, 1/2-\varepsilon\}$ and $\sigma=\min\{s,1\}$ we have $z-e^{it\partial_x^2}z_0 \in C([0,T];H^{s+a}(\mathbb{T}))$ with
\eee{\label{smoothing-est}
\big\| z-e^{it\partial_x^2}z_0\big\|_{C([0,T];H^{s+a}(\mathbb{T}))}
\leq
C(s, \no{z_0}_{H^{\sigma}(\mathbb{T})}, T)\no{z_0}_{H^s(\mathbb{T})}.
}
\end{theorem}
\begin{remark}
Corollary \ref{global} and Lemma \ref{u-z} imply that $z$ satisfies \eqref{duhamel-z} globally.
\end{remark}
We note that the dispersion on the circle is weaker than on the line in the sense that no Kato smoothing or maximal inequalities are available on the circle. Thus, proving \eqref{smoothing-est} requires a careful treatment of resonant frequencies in addition to the differentiation by parts mentioned above.

The nonlinear smoothing estimate \eqref{smoothing-est} allows us to demonstrate the following polynomial-in-time bound, which is the main result of this work.
\begin{theorem}[\b{Polynomial bound}]
\label{polynomial-bound}
Let $s\geq 1$. Then, the global solution $u$ to the periodic dNLS Cauchy problem \eqref{cauchy-dnls} given by Corollary \ref{global} satisfies
\eee{\label{poly-bound}
\no{u(t)}_{H^s(\mathbb{T})}
\leq
C(\varepsilon, s, \no{u_0}_{H^s(\mathbb{T})}, T) \
\langle{t\rangle}^{2(s-1)+\varepsilon},
}
for all $t\in\mathbb{R}$ and $0<\varepsilon\ll 1/2$.
\end{theorem}
Bourgain \cite{b1996, b1997} was the first to demonstrate the connection between nonlinear smoothing and polynomial bounds for Hamiltonian equations. By employing Fourier truncation operators in conjunction with smoothing estimates, he obtained the following local-in-time inequality for solutions of various dispersive equations:
\begin{equation}\label{eq:bour}
\no{ u(t + \delta) }_{H^s} \leqslant \no{u(t)}_{H^s} + C\no{u(t)}_{H^s}^{1-\delta}
\end{equation}
 for some $\delta \in (0,1)$.  Local time iterations using the above inequality resulted in the polynomial growth bound $\no{u(t)}_{H^s} \lesssim \left\langle t \right\rangle^{1/\delta}$.  Staffilani \cite{s1997b, s1997a} used further multilinear smoothing estimates to obtain \eqref{eq:bour}, which led to polynomial bounds of $H^s$ solutions, $s>1$, for Korteweg-de Vries (KdV) and nonlinear Schr\"odinger (NLS) equations.  Colliander, Keel, Staffilani, Takaoka and Tao \cite{ckstt2003} developed a new method using modified energy called the ``upside-down $I$-method'' to produce polynomial bounds in low Sobolev norms, $s\in (0,1)$, for the NLS equation.  Sohinger \cite{s2011a, s2011b} further developed the upside-down $I$-method to obtain polynomial bounds for high Sobolev norms for NLS.  We also refer the reader to \cite{cko2012} and the references therein for further developments in this direction. In addition, Oh and Stefanov \cite{os2020} determined a nonlinear smoothing effect for periodic, generalized KdV equations that gave rise to a polynomial bound in $H^s$, $s>1$. Finally, in \cite{imos2020}, Oh and the authors of the present work identified a nonlinear smoothing effect for a periodic, gauge-transformed Benjamin-Ono equation which led to a polynomial bound on solutions to the periodic Benjamin-Ono equation for $1/2<s\leq 1$. 

Several recent works have established uniform-in-time bounds for a number of completely integrable dispersive equations using inverse scattering techniques. In particular, Killip, Visan and Zhang \cite{kvz2018} showed that the $H^s$-norm of solutions to the KdV and NLS equations is uniformly bounded in time for $-1 \leqslant s < 1$ and $-1/2 < s < 1$, respectively, both on the line and on the circle. Similarly, Koch and Tataru \cite{kt2018} showed that there exists a conserved energy equivalent to the $H^s$-norm for $s>-1/2$ in the case of the NLS and mKdV equations and for $s\geqslant -1$ in the case of the KdV equation. For the Benjamin-Ono equation,  Talbut \cite{t2019} proved a uniform-in-time bound in $H^s$ for $-1/2< s < 0$ on the line and the circle. G\'erard, Kappeler and Topalov \cite{gkt2020} then established this uniform bound for the periodic Benjamin-Ono equation with $s>0$. In the case of dNLS, uniform-in-time bounds were obtained on the line and the circle by Klaus and Schippa \cite{ks2020} for $0<s<1/2$. Furthermore, Bahouri and Perelman \cite{bp2020} showed boundedness of $H^{1/2}(\mathbb{R})$ solutions.

Nonlinear smoothing properties analogous to the one of Theorem \ref{dbp-wp} have been previously established for several important dispersive equations. Indicatively, we mention the works of Erdogan and Tzirakis on the periodic KdV equation \cite{et2013a}  as well as  on the fractional NLS equation on the circle and line \cite{egt2019}, the NLS equation on the half-line \cite{et2016}, the dNLS equation on the line and half-line \cite{egt2018}, and the Zakharov system on the circle \cite{et2013b}.
The main technique used in the proof of these results is known as the normal form method or, as previously mentioned, the differentiation by parts method. It was first introduced by Shatah \cite{s1985} in the context of the Klein-Gordon equation with a quadratic nonlinearity and was further developed  by Germain, Masmoudi and Shatah  for two-dimensional quadratic Schr\"odinger equations \cite{gms2009a} and for the gravity water waves equation \cite{gms2009b}. Babin, Ilyin and Titi \cite{bit2011} applied this method to obtain unconditional well-posedness results for the periodic KdV equation. Chung, Guo, Kwon and Oh \cite{cgko2017} also obtained unconditional well-posedness of the quadratic dNLS using normal form reductions. An alternative formulation of the normal form method, which involves a multilinear, pseudo-differential operator in place of differentiation by parts, was provided by Oh and Stefanov \cite{os2012, o2013} for establishing smoothing estimates and well-posedness.
\vskip 3mm
\noindent
\textbf{Structure of the paper}. In Section \ref{s-dbp}, we employ the gauge transformation \eqref{z}  to remove the resonant frequencies from the $w$-equation \eqref{w-pde}   and perform differentiation by parts on the resulting gauge problem \eqref{gauge-ivp} for $z$ in order to  establish (formally) the Duhamel form \eqref{duhamel-z}. In Section~\ref{s-apriori}, we prove a number of useful a priori estimates for  \eqref{duhamel-z} which are key to establishing Theorem~\ref{dbp-wp}. In Section~\ref{proof1}, we utilize the aforementioned estimates to complete the proof of Theorem~\ref{dbp-wp}. Finally, in Section~\ref{proof2}, we employ the nonlinear smoothing effect from Theorem~\ref{dbp-wp} in order to prove the polynomial bound of Theorem~\ref{polynomial-bound}.
\section{Removal of Resonances and Differentiation by Parts}\label{s-dbp}

We begin with equation \eqref{w-pde} for $w$ and proceed with formal computations. First, note that
\eee{\label{first-nonlin}
-w^2 \overline{w}_x
=
\frac{i}{(2\pi)^{3/2}}\sum_{\xi_1,\xi_2,\xi_3} e^{i(\xi_1-\xi_2+\xi_3)x} \xi_2 \widehat{w}(\xi_1)\overline{\widehat{w}(\xi_2)}\widehat{w}(\xi_3).
}
The resonant frequencies in \eqref{first-nonlin} are associated with $\{\xi_1=\xi_2\}\cup\{\xi_2=\xi_3\}$. Hence, we split \eqref{first-nonlin} into resonant and nonresonant frequencies as follows:
\ddd{
-w^2 \overline{w}_x
&=
NR(-w^2\overline{w}_x)
+
\frac{i}{(2\pi)^{3/2}}\sum_{\xi_1=\xi_2} e^{i(\xi_1+\xi_2-\xi_3)x} \xi_2\widehat{w}(\xi_1)\overline{\widehat{w}(\xi_2)}\widehat{w}(\xi_3)
\nn\\
&\hspace{2.7cm}+\frac{i}{(2\pi)^{3/2}}\sum_{\xi_2=\xi_3} e^{i(\xi_1+\xi_2-\xi_3)x}\xi_2 \widehat{w}(\xi_1)\overline{\widehat{w}(\xi_2)}\widehat{w}(\xi_3)
\nn\\
&=
NR(-w^2\overline{w}_x)+\frac{2iw}{2\pi}\sum_{\xi} \xi |\widehat{w}(\xi)|^2,
\nn}
where we have used the symmetry in $\xi_1$ and $\xi_3$ and 
\eee{\label{NR-def-0}
NR(-w^2\overline{w_x})
=
\frac{i}{(2\pi)^{3/2}}
\sum_{\xi_1\neq \xi_2\neq \xi_3}
e^{i(\xi_1-\xi_2+\xi_3)x} \xi_2 \widehat{w}(\xi_1)\overline{\widehat{w}(\xi_2)}\widehat{w}(\xi_3).
} 
Next, we claim that
$$
\sum_{\xi} \xi |\widehat{w}(\xi)|^2
=
-\int_0^{2\pi}  \text{Im}(\overline{w}_x w)(\theta, t)~d\theta.
$$
Indeed, we have
\ddd{
-\int_0^{2\pi}  \text{Im}(\overline{w}_x w)(\theta, t)~d\theta
&=
\text{Im}\int_0^{2\pi}\frac{i}{2\pi}\sum_{\xi_1,\xi_2} e^{i(\xi_1-\xi_2)\theta}\xi_2\widehat{w}(\xi_1)\overline{\widehat{w}}(\xi_2)~d\theta
\nn\\
&=
\text{Im}\sum_{\xi_1,\xi_2}\int_0^{2\pi}\frac{i}{2\pi} e^{i(\xi_1-\xi_2)\theta}\xi_2\widehat{w}(\xi_1)\overline{\widehat{w}}(\xi_2)~d\theta
\nn\\
&=
\text{Im}\sum_{\xi_1=\xi_2} i\xi_2\widehat{w}(\xi_1)\overline{\widehat{w}}(\xi_2)
=
\text{Im}\sum_{\xi} i\xi|\widehat{w}(\xi)|^2
=
\sum_{\xi}\xi|\widehat{w}(\xi)|^2.
\nn
}
Therefore, the resonant term associated with $-w^2\overline{w}_x$ cancels out and  equation  \eqref{w-pde}  for $w$ becomes
\eee{\label{w-pde2}
w_t-iw_{xx}
=
NR(-w^2\overline{w}_x)
+\frac{i}{2}|w|^4 w-i\mu|w|^2 w
+iw\left(-\frac{1}{4\pi}\int_0^{2\pi} |w|^4(\theta,t)~d\theta+\mu^2\right).
}

We now examine the quintic term
\ddd{
\frac{i}{2}|w|^4 w
=
\frac{i}{2}w\overline{w}w\overline{w}w
=
\frac{i}{2(2\pi)^{5/2}}\sum_{\xi_1,...,\xi_5} e^{i(\xi_1-\xi_2+\xi_3-\xi_4+\xi_5)x}  \widehat{w}(\xi_1)\overline{\widehat{w}(\xi_2)} \widehat{w}(\xi_3)\overline{\widehat{w}(\xi_4)}\widehat{w}(\xi_5).
\nn}
Here, the undesirable frequencies we wish to isolate are $\{\xi_1-\xi_2+\xi_3-\xi_4=0\}\cup\{-\xi_2+\xi_3-\xi_4+\xi_5=0\}\cup\{\xi_1-\xi_2-\xi_4+\xi_5=0\}$. We note that these sets are not necessarily resonant. Thus, we will rewrite the above quintic term as 
\ddd{
\frac{i}{2}|w|^4 w
&=
\frac{i}{2(2\pi)^{5/2}}\sum_{\xi_1,...,\xi_5} e^{i(\xi_1-\xi_2+\xi_3-\xi_4+\xi_5)x} \widehat{w}(\xi_1)\overline{\widehat{w}(\xi_2)} \widehat{w}(\xi_3)\overline{\widehat{w}(\xi_4)}\widehat{w}(\xi_5)
\nn\\
&:=
\mathcal{A}(w)+\frac{3iw}{2(2\pi)^2}\sum_{\xi_1-\xi_2+\xi_3-\xi_4=0} \widehat{w}(\xi_1)\overline{\widehat{w}(\xi_2)}\widehat{w}(\xi_3)\overline{\widehat{w}(\xi_4)},
\nn
}
where we have once again used symmetry and 
\eee{\label{A-def}
\mathcal{A}(w)
=
\frac{i}{2(2\pi)^{5/2}}
\sum_{\substack{\xi_2+\xi_4\neq \xi_1+\xi_3 \\ \xi_2+\xi_4\neq \xi_1+\xi_5 \\ \xi_2+\xi_4\neq \xi_3+\xi_5}}
e^{i(\xi_1-\xi_2+\xi_3-\xi_4+\xi_5)x} \widehat{w}(\xi_1)\overline{\widehat{w}(\xi_2)} \widehat{w}(\xi_3)\overline{\widehat{w}(\xi_4)}\widehat{w}(\xi_5).
} 

Next, we observe that
\ddd{
-\frac{1}{4\pi}\int_0^{2\pi} |w|^4(\theta,t)~d\theta
&=
-\frac{1}{4\pi}\int_0^{2\pi} \frac{1}{(2\pi)^2}\sum_{\xi_1,...,\xi_4} e^{i(\xi_1-\xi_2+\xi_3-\xi_4)\theta}\widehat{w}(\xi_1)\overline{\widehat{w}(\xi_2)}\widehat{w}(\xi_3)\overline{\widehat{w}(\xi_4)}~d\theta
\nn\\
&=
-\frac{1}{2(2\pi)^2}\sum_{\xi_1-\xi_2+\xi_3-\xi_4=0} \widehat{w}(\xi_1)\overline{\widehat{w}(\xi_2)}\widehat{w}(\xi_3)\overline{\widehat{w}(\xi_4)}.
\nn
}
Thus, we may write \eqref{w-pde2} as
\ddd{
w_t-iw_{xx}
&=
NR(-w^2\overline{w}_x)
+\mathcal{A}(w)-i\mu|w|^2 w
+\frac{iw}{(2\pi)^2}\sum_{\xi_1-\xi_2+\xi_3-\xi_4=0} \widehat{w}(\xi_1)\overline{\widehat{w}(\xi_2)}\widehat{w}(\xi_3)\overline{\widehat{w}(\xi_4)}+i\mu^2 w
\nn\\
&=
NR(-w^2\overline{w}_x)
+\mathcal{A}(w)
-i\mu|w|^2 w
+\frac{iw}{2\pi}\int_0^{2\pi}|w(\theta,t)|^4~d\theta+i\mu^2 w.
\nn}
\indent 
Finally, we recognize that $|w|^2 w$  contains the same resonant frequencies as $-w^2\overline{w}_x$. Thus, writing 
\ddd{
|w|^2 w
&=
NR(|w|^2 w)+\frac{2w}{2\pi}\sum_{\xi}|\widehat{w}(\xi)|^2
\nn\\
&=
NR(|w|^2 w)+\frac{2w}{2\pi}\no{w}_{L^2(\mathbb{T})}^2
=
NR(|w|^2 w)+2\mu w
\nn
}
where, analogously to \eqref{NR-def-0}, 
\eee{\label{NR-def}
NR(|w|^2 w)
=
\frac{1}{(2\pi)^{3/2}}
\sum_{\xi_1\neq \xi_2\neq \xi_3}
e^{i(\xi_1-\xi_2+\xi_3)x}  \widehat{w}(\xi_1)\overline{\widehat{w}(\xi_2)}\widehat{w}(\xi_3), 
}
we overall conclude that $w$ formally satisfies the equation
\eee{
w_t-iw_{xx}
=
NR(-w^2\overline{w}_x)
+\mathcal{A}(w)
-i\mu NR(|w|^2w)
+\frac{iw}{2\pi}\int_0^{2\pi}|w(\theta,t)|^4~d\theta-i\mu^2 w.
}
In turn, recalling the second gauge transformation \eqref{z}, %$z=e^{-ig(t)}w$ with $g(t)=\int_0^t (\frac{8\pi^3-1}{16\pi^4})\no{w(t')}_{L^4(\mathbb{T})}^4 dt'-\mu^2 t$. 
we deduce that $z$ satisfies
\begin{subequations}\label{gauge-ivp}
\ddd{
&z_t-iz_{xx}
=
NR(-z^2 \overline{z}_x)+\mathcal{A}(z)-i\mu NR(|z|^2 z)
+\frac{i}{16\pi^4}\no{z(t)}_{L^4(\mathbb{T})}^4 z, \quad x\in\mathbb{T}, \ t\in\mathbb{R},
\label{gauge-eq}
\\
&z(x,0)
=
z_0(x)\in H^s(\mathbb{T}).
\label{gauge-id}
}
\end{subequations}
Next, we proceed with differentiation by parts at the interaction representation level $e^{it\partial_x^2}z$ of equation \eqref{gauge-eq} in order to determine a nonlinear smoothing effect for the solution $z$ emanating from Theorem \ref{herr-wp}.

First, we isolate the more troublesome terms associated with \eqref{gauge-eq} and then we apply differentiation by parts to these terms. Note that
\ddd{\label{gauge-ft}
\partial_t (e^{it\xi^2}\widehat{z}(\xi))
&=
\frac{1}{(2\pi)^{3/2}}
\sum_{\substack{\xi_1-\xi_2+\xi_3=\xi \\ \xi_1\neq\xi_2\neq\xi_3}} 
e^{it\xi^2} i\xi_2 \widehat{z}(\xi_1)\overline{\widehat{z}(\xi_2)}\widehat{z}(\xi_3)
+e^{it\xi^2}\widehat{\mathcal{A}(z)}(\xi)
\nn\\
&\quad
-i\mu e^{it\xi^2}\mathcal F_x\!\left(NR(|z|^2z)\right)(\xi)
+\frac{i}{16\pi^4}\no{z(t)}_{L^4(\mathbb{T})}^4 e^{it\xi^2} \widehat{z}(\xi).
}
Next, by symmetry, observe that
\ddd{
\sum_{\substack{\xi_1-\xi_2+\xi_3=\xi \\ \xi_1\neq\xi_2\neq\xi_3}} 
e^{it\xi^2} i\xi_2 \widehat{z}(\xi_1)\overline{\widehat{z}(\xi_2)}\widehat{z}(\xi_3)
&=
2\sum_{\substack{\xi_1-\xi_2+\xi_3=\xi \\ \xi_1\neq\xi_2\neq\xi_3 \\ |\xi_1|\geq |\xi_3|}} 
e^{it\xi^2} i\xi_2 \widehat{z}(\xi_1)\overline{\widehat{z}(\xi_2)}\widehat{z}(\xi_3).
\nn
}
Letting 
\eee{
S_{\xi} := \left\{\xi_1-\xi_2+\xi_3=\xi\right\}\cap \left\{\xi_1\neq\xi_2\neq\xi_3\right\} \cap \left\{|\xi_1|\geq |\xi_3|\right\},
}
we have
\ddd{
\sum_{S_{\xi}} 
e^{it\xi^2} i\xi_2 \widehat{z}(\xi_1)\overline{\widehat{z}(\xi_2)}\widehat{z}(\xi_3)
&=
\sum_{\substack{S_{\xi} \\ |\xi_1|\gg |\xi_2|}} e^{it\xi^2} i\xi_2 \widehat{z}(\xi_1)\overline{\widehat{z}(\xi_2)}\widehat{z}(\xi_3)
+
\sum_{\substack{S_{\xi} \\ |\xi_1|\sim|\xi_2|\gg|\xi_3|}} e^{it\xi^2} i\xi_2 \widehat{z}(\xi_1)\overline{\widehat{z}(\xi_2)}\widehat{z}(\xi_3) 
\nn\\
&\quad + 
\sum_{\substack{S_{\xi} \\ |\xi_1|\sim|\xi_2|\sim|\xi_3|}} e^{it\xi^2} i\xi_2 \widehat{z}(\xi_1)\overline{\widehat{z}(\xi_2)}\widehat{z}(\xi_3) 
+
\sum_{\substack{S_{\xi} \\ |\xi_1|\ll |\xi_2|}} e^{it\xi^2} i\xi_2 \widehat{z}(\xi_1)\overline{\widehat{z}(\xi_2)}\widehat{z}(\xi_3)
\nn\\
&=:
e^{it\xi^2}
(2\pi)^{3/2}
\sum_{\ell=1}^3 \widehat{\mathcal{B}_{\ell}(z)}(\xi)+\sum_{\substack{S_{\xi} \\ |\xi_1|\ll |\xi_2|}} e^{it\xi^2} i\xi_2 \widehat{z}(\xi_1)\overline{\widehat{z}(\xi_2)}\widehat{z}(\xi_3),\nn
}
where
\sss{\label{B-def}
\ddd{
\widehat{\mathcal{B}_1(z)}(\xi) &=\frac{1}{(2\pi)^{3/2}}\sum_{\substack{S_{\xi} \\ |\xi_1|\gg |\xi_2|}} i\xi_2 \widehat{z}(\xi_1)\overline{\widehat{z}(\xi_2)}\widehat{z}(\xi_3),
\\
\widehat{\mathcal{B}_2(z)}(\xi) &=\frac{1}{(2\pi)^{3/2}} \sum_{\substack{S_{\xi} \\ |\xi_1|\sim|\xi_2|\gg|\xi_3|}} i\xi_2 \widehat{z}(\xi_1)\overline{\widehat{z}(\xi_2)}\widehat{z}(\xi_3),\\
\widehat{\mathcal{B}_3(z)}(\xi) &=\frac{1}{(2\pi)^{3/2}}\sum_{\substack{S_{\xi} \\ |\xi_1|\sim|\xi_2|\sim|\xi_3|}}i\xi_2 \widehat{z}(\xi_1)\overline{\widehat{z}(\xi_2)}\widehat{z}(\xi_3).
}
}

Thus, it follows that
\ddd{
\partial_t (e^{it\xi^2}\widehat{z}(\xi))
&=
e^{it\xi^2} \bigg(2\sum_{\ell=1}^3 \widehat{\mathcal{B}_{\ell}(z)}(\xi) +\widehat{\mathcal{A}(z)}(\xi)+i\mu \, \mathcal F_x\!\left(NR(|z|^2 z)\right)(\xi) +\frac{i}{16\pi^4}\no{z(t)}_{L^4(\mathbb{T})}^4 \widehat{z}(\xi)\bigg)
\nn\\
&\quad
+
\frac{1}{(2\pi)^{3/2}}
\sum_{\substack{S_{\xi} \\ |\xi_1|\ll |\xi_2|}} e^{it\xi^2} i\xi_2 \widehat{z}(\xi_1)\overline{\widehat{z}(\xi_2)}\widehat{z}(\xi_3).\label{gauge-ft2}
}
We will apply differentiation by parts on the last term of \eqref{gauge-ft2}, as this term does not necessarily gain derivatives via the Fourier restriction norm method. First, observe that
\eee{
\sum_{\substack{S_{\xi} \\ |\xi_1|\ll |\xi_2|}} e^{it\xi^2} i\xi_2 \widehat{z}(\xi_1)\overline{\widehat{z}(\xi_2)}\widehat{z}(\xi_3)
=
\sum_{\substack{S_{\xi} \\ |\xi_1|\ll |\xi_2|}} e^{2it(\xi_2-\xi_1)(\xi_2-\xi_3)} i\xi_2 (e^{it\xi_1^2}\widehat{z}(\xi_1))\overline{(e^{it\xi_2^2}\widehat{z}(\xi_2))}(e^{it\xi_3^2}\widehat{z}(\xi_3)).\nn
}
Let 
\eee{
\Psi=2(\xi_2-\xi_1)(\xi_2-\xi_3).
}
Applying differentiation by parts yields
\ddd{
&\quad
\frac{1}{(2\pi)^{3/2}}
\sum_{\substack{S_{\xi} \\ |\xi_1|\ll |\xi_2|}} e^{2it(\xi_2-\xi_1)(\xi_2-\xi_3)} i\xi_2 (e^{it\xi_1^2}\widehat{z}(\xi_1))\overline{(e^{it\xi_2^2}\widehat{z}(\xi_2))}(e^{it\xi_3^2}\widehat{z}(\xi_3))
\nn\\
&=:
\partial_t \left[e^{it\xi^2}\widehat{NF(z)}(\xi)\right] -\sum_{\ell=1}^3\widehat{\mathcal{N}_{\ell}(z)}(\xi),\nn
}
where
\ddd{
\widehat{NF(z)}(\xi) &= 
\frac{1}{(2\pi)^{3/2}}
\sum_{\substack{S_{\xi} \\ |\xi_1|\ll |\xi_2|}} 
\frac{\xi_2}{\Psi} \, \widehat{z}(\xi_1)\overline{\widehat{z}(\xi_2)}\widehat{z}(\xi_3), 
\label{NF-def}
\\
\widehat{\mathcal{N}_1 (z)}(\xi) &=
\frac{1}{(2\pi)^{3/2}}
\sum_{\substack{S_{\xi} \\ |\xi_1|\ll |\xi_2|}} 
e^{it\Psi} \, \frac{\xi_2}{\Psi} \, \partial_t(e^{it\xi_1^2}\widehat{z}(\xi_1))\overline{(e^{it\xi_2^2}\widehat{z}(\xi_2))}(e^{it\xi_3^2}\widehat{z}(\xi_3)), 
}
and $\mathcal{N}_2(z)$, $\mathcal{N}_3(z)$ are given by analogous expressions where the time derivative is applied on the second and on the third term respectively. 
Note that, although the analysis for $\mathcal{N}_3(z)$ is essentially the same with the one for $\mathcal{N}_1(z)$,   its inclusion is  necessary since the symmetry in $\xi_1$ and $\xi_3$ has been broken.

Inserting \eqref{gauge-ft} into $\mathcal{N}_1(z)$, we obtain
\ddd{
\widehat{\mathcal{N}_1 (z)}(\xi) 
%&=
%\frac{1}{(2\pi)^{3/2}}
%\sum_{\substack{S_{\xi} \\ |\xi_1|\ll |\xi_2|}} 
%e^{it\Psi} \, \frac{\xi_2}{\Psi} \, \partial_t(e^{it\xi_1^2}\widehat{z}(\xi_1))\overline{(e^{it\xi_2^2}\widehat{z}(\xi_2))}(e^{it\xi_3^2}\widehat{z}(\xi_3))
%\nn\\
&=
\frac{1}{8\pi^3}
\sum_{\substack{S_{\xi} \\ |\xi_1|\ll |\xi_2|}} 
e^{it\Psi} \, \frac{\xi_2}{\Psi} \,
\bigg(\sum_{\substack{\xi_4-\xi_5+\xi_6=\xi_1 \\ \xi_4\neq\xi_5\neq\xi_6}} 
e^{it\xi_1^2} i\xi_5 \widehat{z}(\xi_4)\overline{\widehat{z}(\xi_5)}\widehat{z}(\xi_6)\bigg)
\overline{(e^{it\xi_2^2}\widehat{z}(\xi_2))}(e^{it\xi_3^2}\widehat{z}(\xi_3))
\nn\\
&\quad
+
\frac{1}{(2\pi)^{3/2}}
\sum_{\substack{S_{\xi} \\ |\xi_1|\ll |\xi_2|}} 
e^{it\xi^2}\, \frac{\xi_2}{\Psi} \, 
\widehat{\mathcal{A}(z)}(\xi_1)
\overline{\widehat{z}(\xi_2)}\widehat{z}(\xi_3)
\nn\\
&\quad
-\frac{i\mu}{(2\pi)^{3/2}}
\sum_{\substack{S_{\xi} \\ |\xi_1|\ll |\xi_2|}} 
e^{it\xi^2} \, \frac{\xi_2}{\Psi} \, 
\mathcal F_x\!\left(NR(|z|^2z)\right)(\xi_1)
\overline{\widehat{z}(\xi_2)}\widehat{z}(\xi_3) \nn\\
&\quad
+
\frac{i}{(2\pi)^{11/2}}\no{z(t)}_{L^4(\mathbb{T})}^4
\sum_{\substack{S_{\xi} \\ |\xi_1|\ll |\xi_2|}} 
e^{it\Psi} \, \frac{\xi_2}{\Psi} \, (e^{it\xi_1^2} \widehat{z}(\xi_1))\overline{(e^{it\xi_2^2}\widehat{z}(\xi_2))}(e^{it\xi_3^2}\widehat{z}(\xi_3))
\nn\\
&=:
e^{it\xi^2}\sum_{\ell=1}^3 \widehat{\mathcal{N}_{1,\ell}(z)}(\xi) 
+\frac{i}{16\pi^4} \no{z(t)}_{L^4(\mathbb{T})}^4 e^{it\xi^2} \widehat{NF(z)}(\xi).
\label{N1l-def}
}
Similarly, for $\mathcal{N}_2(z)$  we find
\ddd{
\widehat{\mathcal{N}_2 (z)}(\xi) 
&=
%\frac{1}{(2\pi)^{3/2}}
%\sum_{\substack{S_{\xi} \\ |\xi_1|\ll |\xi_2|}} 
%e^{it\Psi}\frac{\xi_2}{\Psi} (e^{it\xi_1^2}\widehat{z}(\xi_1))\overline{\partial_t(e^{it\xi_2^2}\widehat{z}(\xi_2))}(e^{it\xi_3^2}\widehat{z}(\xi_3))\nn\\
%&=
\frac{1}{8\pi^3}
\sum_{\substack{S_{\xi} \\ |\xi_1|\ll |\xi_2|}} 
e^{it\Psi} \, \frac{\xi_2}{\Psi} \, 
(e^{it\xi_1^2}\widehat{z}(\xi_1))
\overline{
\bigg(\sum_{\substack{\xi_4-\xi_5+\xi_6=\xi_2 \\ \xi_4\neq\xi_5\neq\xi_6}} 
e^{it\xi_2^2} i\xi_5 \widehat{z}(\xi_4)\overline{\widehat{z}(\xi_5)}\widehat{z}(\xi_6)\bigg)
}
(e^{it\xi_3^2}\widehat{z}(\xi_3))\nn\\
&\quad
+
\frac{1}{(2\pi)^{3/2}}
\sum_{\substack{S_{\xi} \\ |\xi_1|\ll |\xi_2|}} 
e^{it\xi^2} \, \frac{\xi_2}{\Psi} \, 
\widehat{z}(\xi_1)
\overline{\widehat{\mathcal{A}(z)}(\xi_2)}
\widehat{z}(\xi_3)
\nn\\
&\quad
+\frac{i\mu}{(2\pi)^{3/2}}
\sum_{\substack{S_{\xi} \\ |\xi_1|\ll |\xi_2|}} 
e^{it\xi^2} \, \frac{\xi_2}{\Psi} \, 
\widehat{z}(\xi_1)
\overline{\mathcal F_x\!\left(NR(|z|^2z)\right)(\xi_2)}
\widehat{z}(\xi_3) 
\nn\\
&\quad
-
\frac{i}{(2\pi)^{11/2}}\no{z(t)}_{L^4(\mathbb{T})}^4
\sum_{\substack{S_{\xi} \\ |\xi_1|\ll |\xi_2|}} 
e^{it\Psi} \, \frac{\xi_2}{\Psi} \, (e^{it\xi_1^2} \widehat{z}(\xi_1))\overline{(e^{it\xi_2^2}\widehat{z}(\xi_2))}(e^{it\xi_3^2}\widehat{z}(\xi_3))
\nn\\
&=:
e^{it\xi^2}\sum_{\ell=1}^3 \widehat{\mathcal{N}_{2,\ell}(z)}(\xi)
-
\frac{i}{16\pi^4} \no{z(t)}_{L^4(\mathbb{T})}^4 e^{it\xi^2} \widehat{NF(z)}(\xi).
\label{N2l-def}
}
Further expanding these terms, we see that
\ddd{
\widehat{\mathcal{N}_{1,1}(z)}(\xi)
&=
\frac{1}{8\pi^3}
\sum_{N_{1,1}(\xi)}
\frac{i\xi_2 \xi_4}{\Psi}
\widehat{z}(\xi_1)\overline{\widehat{z}(\xi_2)}\widehat{z}(\xi_3)\overline{\widehat{z}(\xi_4)}\widehat{z}(\xi_5),\nn \\
\widehat{\mathcal{N}_{1,2}(z)}(\xi)
&=
\frac{1}{8\pi^3}
\sum_{N_{1,2}(\xi)}
\frac{i\xi_6}{\Psi}
\widehat{z}(\xi_1)\overline{\widehat{z}(\xi_2)}\widehat{z}(\xi_3)\overline{\widehat{z}(\xi_4)}\widehat{z}(\xi_5)\overline{\widehat{z}(\xi_6)}\widehat{z}(\xi_7)\nn\\
\widehat{\mathcal{N}_{2,1}(z)}(\xi)
&=
\frac{1}{8\pi^3}
\sum_{N_{2,1}(\xi)}
\frac{-i(\xi_2-\xi_3+\xi_4)\xi_3}{\Psi}
\widehat{z}(\xi_1)\overline{\widehat{z}(\xi_2)}\widehat{z}(\xi_3)\overline{\widehat{z}(\xi_4)}\widehat{z}(\xi_5),\nn\\
\widehat{\mathcal{N}_{2,2}(z)}(\xi)
&=
\frac{1}{8\pi^3}
\sum_{N_{2,2}(\xi)}
\frac{i(\xi_2-\xi_3+\xi_4-\xi_5+\xi_6)}{\Psi}
\widehat{z}(\xi_1)\overline{\widehat{z}(\xi_2)}\widehat{z}(\xi_3)\overline{\widehat{z}(\xi_4)}\widehat{z}(\xi_5)\overline{\widehat{z}(\xi_6)}\widehat{z}(\xi_7),\nn
}
where the sets $N_{1,1}, N_{1,2}, N_{2,1}$ and $N_{2,2}$ are given by
\ddd{
N_{1,1}(\xi)
&=
\{\xi_1-\xi_2+\xi_3-\xi_4+\xi_5=\xi\} \cap \{\xi_1-\xi_2+\xi_3\neq \xi_4\neq \xi_5\}
\nn\\
&\hspace{0.5cm}\cap \{\xi_1\neq\xi_2\neq\xi_3\} \cap \{|\xi_1-\xi_2+\xi_3|\ll |\xi_4|\} 
\cap \{|\xi_1|\geq |\xi_5|\},
\nn\\
N_{1,2}(\xi)
&=
\{\xi_1-\xi_2+\xi_3-\xi_4+\xi_5-\xi_6+\xi_7=\xi\} \cap \{\xi_1-\xi_2+\xi_3-\xi_4+\xi_5\neq \xi_6 \neq \xi_7\}
\nn\\
&\hspace{0.5cm}\cap \{\xi_2+\xi_4\neq \xi_1+\xi_3, \xi_1+\xi_5, \xi_3+\xi_5\}\cap \{|\xi_1-\xi_2+\xi_3+\xi_4+\xi_5|\ll |\xi_6|\}
\cap \{|\xi_1|\geq |\xi_7|\},
\nn\\
N_{2,1}(\xi)
&=
\{\xi_1-\xi_2+\xi_3-\xi_4+\xi_5=\xi\} \cap \{\xi_1\neq \xi_2-\xi_3+\xi_4\neq \xi_5\} \nn\\
&\hspace{0.5cm}\cap \{\xi_2\neq\xi_3\neq\xi_4\} \cap \{|\xi_1|\ll |\xi_2-\xi_3+\xi_4|\}
\cap \{|\xi_1|\geq |\xi_5|\},
\nn\\
N_{2,2}(\xi)
&=
\{\xi_1-\xi_2+\xi_3-\xi_4+\xi_5-\xi_6+\xi_7=\xi\} \cap \{\xi_1\neq \xi_2-\xi_3+\xi_4-\xi_5+\xi_6\neq \xi_7\}
\nn\\
&\hspace{0.5cm} \cap \{\xi_3+\xi_5\neq \xi_2+\xi_4, \xi_2+\xi_6, \xi_4+\xi_6\} \cap \{|\xi_1|\ll |\xi_2-\xi_3+\xi_4-\xi_5+\xi_6|\} \cap\{|\xi_1|\geq |\xi_7|\}.
\nn
}
The analysis of $\mathcal{N}_{1,1}(z)$ and $\mathcal{N}_{2,1}(z)$ covers the one of $\mathcal{N}_{1,3}(z)$ and $\mathcal{N}_{2,3}(z)$, respectively; therefore, the analysis of the latter two terms  is omitted. 

Finally, we note that $\mathcal{N}_{2,1}(z)$ possesses the unfavorable frequency interaction $\xi_2+\xi_4=\xi_1+\xi_5$. To handle this, we  decompose $\mathcal{N}_{2,1}(z)$ as 
\ddd{
\widehat{\mathcal{N}_{2,1}(z)}(\xi)
&=
\frac{1}{8\pi^3}
\sum_{\substack{N_{2,1}(\xi) \\ \xi_2+\xi_4\neq \xi_1+\xi_5}}
\frac{-i(\xi_2-\xi_3+\xi_4)\xi_3}{\Psi} 
\widehat{z}(\xi_1)\overline{\widehat{z}(\xi_2)}\widehat{z}(\xi_3)\overline{\widehat{z}(\xi_4)}\widehat{z}(\xi_5) \nn\\
&\quad
+
\frac{1}{8\pi^3}
\sum_{\substack{{N_{2,1}(\xi)}\\ \xi_2+\xi_4=\xi_1+\xi_5}}
\frac{-i(\xi_2-\xi_3+\xi_4)\xi_3}{\Psi}
\widehat{z}(\xi_1)\overline{\widehat{z}(\xi_2)}\widehat{z}(\xi_3)\overline{\widehat{z}(\xi_4)}\widehat{z}(\xi_5)
\nn\\
&=:
\widehat{\mathcal{N}_{2,1}^*(z)}(\xi)
+
\frac{1}{8\pi^3}
\sum_{\substack{{N_{2,1}(\xi)}\\ \xi_2+\xi_4=\xi_1+\xi_5}} \frac{i(\xi-\xi_1-\xi_5)\xi}{(\xi-\xi_1)(\xi-\xi_5)}
\widehat{z}(\xi_1)\overline{\widehat{z}(\xi_2)}\widehat{z}(\xi)\overline{\widehat{z}(\xi_4)}\widehat{z}(\xi_5)
\nn\\
&=
\widehat{\mathcal{N}_{2,1}^*(z)}(\xi)
+
\frac{1}{8\pi^3}
\sum_{\substack{{N_{2,1}(\xi)}\\ \xi_2+\xi_4=\xi_1+\xi_5}} \frac{i\xi_1 \xi_5}{(\xi-\xi_1)(\xi-\xi_5)}
\widehat{z}(\xi_1)\overline{\widehat{z}(\xi_2)}\widehat{z}(\xi)\overline{\widehat{z}(\xi_4)}\widehat{z}(\xi_5) \nn\\
&\quad
+
\frac{i}{8\pi^3}
\sum_{\substack{{N_{2,1}(\xi)}\\ \xi_2+\xi_4=\xi_1+\xi_5}} \widehat{z}(\xi_1)\overline{\widehat{z}(\xi_2)}\widehat{z}(\xi)\overline{\widehat{z}(\xi_4)}\widehat{z}(\xi_5)
\nn\\
&=:
\sum_{\ell =1}^3 \widehat{\mathcal{N}_{2,\ell}^*(z)}(\xi)
\label{N2l*-def}
}
and further expand $\mathcal{N}^*_{2,3}(z)$ as follows:
\ddd{
\widehat{\mathcal{N}_{2,3}^*(z)}(\xi) 
&=
\frac{i}{8\pi^3}
\widehat{z}(\xi) \sum_{\substack{N_{2,1}(\xi) \\ \xi_1+\xi_3=\xi_2+\xi_4}}
\widehat{z}(\xi_1)\overline{\widehat{z}(\xi_2)} \widehat{z}(\xi_3)\overline{\widehat{z}(\xi_4)}
\nn\\
&=
\frac{i}{8\pi^3}
\widehat{z}(\xi) \sum_{\substack{\xi_1+\xi_3=\xi_2+\xi_4 \\ |\xi_1|, |\xi_2|, |\xi_3|, |\xi_4| \ll |\xi|}}
\widehat{z}(\xi_1)\overline{\widehat{z}(\xi_2)} \widehat{z}(\xi_3)\overline{\widehat{z}(\xi_4)}
\nn\\
&\quad
+
\frac{i}{8\pi^3}
\widehat{z}(\xi) \sum_{\substack{N_{2,1}(\xi) \\ \xi_1+\xi_3=\xi_2+\xi_4 \\ |\xi|\lesssim |\xi_{\ell}|,~\text{for some}~\ell}}
\widehat{z}(\xi_1)\overline{\widehat{z}(\xi_2)} \widehat{z}(\xi_3)\overline{\widehat{z}(\xi_4)}
\nn\\
&=:
\frac{i}{16\pi^4}\widehat{z}(\xi) \int_0^{2\pi} \Big|\sum_{|\xi_1| \ll |\xi|} e^{i\xi_1 \theta} \widehat{z}(\xi_1)\Big|^4~d\theta
+
\widehat{\mathcal{E}_1(z)}(\xi)
\nn\\
&=
\frac{i}{16\pi^4}
\widehat{z}(\xi) \int_0^{2\pi} \Big|z-\sum_{|\xi|\lesssim |\xi_1|} e^{i\xi_1\theta} \widehat{z}(\xi_1)\Big|^4~d\theta
+\widehat{\mathcal{E}_1(z)}(\xi)
\nn\\
&=:
\frac{i}{16\pi^4}
\widehat{z}(\xi) \no{z}_{L^4(\mathbb{T})}^4 + \widehat{\mathcal{E}_1(z)}(\xi) +\widehat{\mathcal{E}_2(z)}(\xi),\label{E-def}
}
where 
\ddd{
-\frac{16\pi^4 i}{\sqrt{2\pi}}
\widehat{\mathcal{E}_2(z)}(\xi)
&=
-2\widehat{z}(\xi)\sum_{|\xi|\lesssim |\xi_1|} \overline{\widehat{z}(\xi_1)} \widehat{|z|^2 z}(\xi_1)
-
2\widehat{z}(\xi)\sum_{|\xi|\lesssim |\xi_1|}\widehat{z}(\xi_1) \overline{\widehat{|z|^2 z}(\xi_1)}
\nn\\
&\quad 
+
4\widehat{z}(\xi) \sum_{|\xi|\lesssim |\xi_1|, |\xi_2|}\widehat{z}(\xi_1)\overline{\widehat{z}(\xi_2)} \, \overline{\widehat{|z|^2}(\xi_1-\xi_2)}
+
\widehat{z}(\xi)\sum_{|\xi|\lesssim |\xi_1|, |\xi_2|} \widehat{z}(\xi_1)\widehat{z}(\xi_2)\overline{\widehat{z^2}(\xi_1+\xi_2)}\nn\\
&\quad 
+
\widehat{z}(\xi) \sum_{|\xi|\lesssim |\xi_1|, |\xi_2|} \overline{\widehat{z}(\xi_1)}\overline{\widehat{z}(\xi_2)} \widehat{z^2}(\xi_1+\xi_2)\nn\\
&\quad
-
2\widehat{z}(\xi)\sum_{|\xi|\lesssim |\xi_1|, |\xi_2|, |\xi_3|} \widehat{z}(\xi_1)\overline{\widehat{z}(\xi_2)} \overline{\widehat{z}(\xi_3)}\widehat{z}(-\xi_1+\xi_2+\xi_3)
\nn\\
&\quad
-
2\widehat{z}(\xi)\sum_{|\xi|\lesssim |\xi_1|, |\xi_2|, |\xi_3|} \widehat{z}(\xi_1)\overline{\widehat{z}(\xi_2)} \widehat{z}(\xi_3)\overline{\widehat{z}(\xi_1-\xi_2+\xi_3)}
\nn\\
&\quad
+
2\pi\widehat{z}(\xi)\sum_{\substack{\xi_1+\xi_3=\xi_2+\xi_4\\ |\xi|\lesssim |\xi_1|, |\xi_2|, |\xi_3|, |\xi_4|}}\widehat{z}(\xi_1)\overline{\widehat{z}(\xi_2)}\widehat{z}(\xi_3)\overline{\widehat{z}(\xi_4)}.
\label{E-def2}
}

Overall, we have  
\ddd{
&\quad
\partial_t (e^{it\xi^2}\widehat{z}(\xi))
=
e^{it\xi^2} \bigg(2\sum_{\ell=1}^3 \widehat{\mathcal{B}_{\ell}(z)}(\xi) +\widehat{\mathcal{A}(z)}(\xi)-i\mu \, \mathcal F_x\!\left(NR(|z|^2 z)\right)(\xi)\bigg)
+
\sum_{\substack{S_{\xi} \\ |\xi_1|\ll |\xi_2|}} e^{it\xi^2} i\xi_2 \widehat{z}(\xi_1)\overline{\widehat{z}(\xi_2)}\widehat{z}(\xi_3)\nn\\
&\hskip 2.8cm
+\frac{i}{16\pi^4}\no{z(t)}_{L^4(\mathbb{T})}^4 e^{it\xi^2} \widehat{z}(\xi)
\nn\\
&=
\partial_t \left[e^{it\xi^2}\widehat{NF(z)}(\xi)\right]
+
e^{it\xi^2} \bigg(2\sum_{\ell=1}^3 \widehat{\mathcal{B}_{\ell}(z)}(\xi) +\widehat{\mathcal{A}(z)}(\xi)-i\mu \, \mathcal F_x\!\left(NR(|z|^2 z)\right)(\xi)
-\sum_{\ell_1, \ell_2=1}^3 \widehat{\mathcal{N}_{\ell_1, \ell_2}(z)}(\xi)
\bigg),
\nn\\
&\hskip 2.8cm
+\frac{i}{16\pi^4}\no{z(t)}_{L^4(\mathbb{T})}^4 e^{it\xi^2} \widehat{z}(\xi)
-\frac{i}{16\pi^4} \no{z(t)}_{L^4(\mathbb{T})}^4 e^{it\xi^2} \widehat{NF(z)}(\xi),
\nn
}
where 
\ddd{
\widehat{\mathcal{N}_{2,1}(z)}(\xi)
&=
\sum_{\ell = 1}^2 \left(\widehat{\mathcal{N}_{2,\ell}^*(z)}(\xi)
+
\widehat{\mathcal{E}_{\ell}(z)}(\xi)
\right)
+
\frac{i}{16\pi^4} \widehat{z}(\xi) \no{z}_{L^4(\mathbb{T})}^4.
\nn
}
Hence, we see that
\ddd{
\partial_t (e^{it\xi^2}\widehat{z}(\xi))
&=
\partial_t \left[e^{it\xi^2}\widehat{NF(z)}(\xi)\right]
\nn\\
&\quad
+
e^{it\xi^2} \bigg(2\sum_{\ell=1}^3 \widehat{\mathcal{B}_{\ell}(z)}(\xi) +\widehat{\mathcal{A}(z)}(\xi)-i\mu \, \mathcal F_x\!\left(NR(|z|^2 z)\right)(\xi)
-\sum_{\substack{\ell_1, \ell_2=1 \\ (\ell_1, \ell_2)\neq (2,1)}}^3 \widehat{\mathcal{N}_{\ell_1, \ell_2}(z)}(\xi)
\bigg)
\nn\\
&\quad
-
e^{it\xi^2} \sum_{\ell = 1}^2 \left(\widehat{\mathcal{N}_{2,\ell}^*(z)}(\xi)
+
\widehat{\mathcal{E}_{\ell}(z)}(\xi)
\right)
-\frac{i}{16\pi^4} \no{z(t)}_{L^4(\mathbb{T})}^4 e^{it\xi^2} \widehat{NF(z)}(\xi).
\label{dbp-z}
}
Finally, we integrate \eqref{dbp-z} on $[0,t]$ and invert the Fourier transform to arrive at
\ddd{\label{duhamel-z}
z(x,t)
&=
e^{it\partial_x^2}z_0(x)
+
NF(z)(x,t)-e^{it\partial_x^2}NF(z_0)(x)+\int_0^t e^{i(t-t')\partial_x^2} N(z)(x,t')~dt',
}
where
\ddd{
\widehat{N(z)}(\xi)
&=
2\sum_{\ell=1}^3 \widehat{\mathcal{B}_{\ell}(z)}(\xi) +\widehat{\mathcal{A}(z)}(\xi)-i\mu \, \mathcal F_x\!\left(NR(|z|^2 z)\right)(\xi)
-\sum_{\substack{\ell_1, \ell_2=1 \\ (\ell_1, \ell_2)\neq (2,1)}}^3 \widehat{\mathcal{N}_{\ell_1, \ell_2}(z)}(\xi)\nn\\
&\quad
-
\sum_{\ell = 1}^2 \left(\widehat{\mathcal{N}_{2,\ell}^*(z)}(\xi)
+
\widehat{\mathcal{E}_{\ell}(z)}(\xi)
\right)
-\frac{i}{16\pi^4} \no{z(t)}_{L^4(\mathbb{T})}^4 e^{it\xi^2} \widehat{NF(z)}(\xi)
.\nn
}
\section{A Priori Estimates}\label{s-apriori}
We begin by recalling certain linear estimates from the literature which are relevant for our analysis. From Bourgain \cite{b1993}, we have the following useful estimates.
\begin{lemma}[\cite{b1993}]
For any $\delta>0$, let $u\in X^{0,\frac{3}{8}+\delta}$. Then, there exists a constant $c>0$ such that
\ddd{\label{l4-est}
\no{u}_{L^4_{t,x}}
&\leq
c
\no{u}_{X^{0,\frac{3}{8}+\delta}}.
}
Furthermore, if $u\in X^{\delta, \frac{1}{2}+\delta}$ then there exists a constant $c>0$ such that
\ddd{\label{l6-est}
\no{u}_{L^6_{t,x}}
&\leq
c
\no{u}_{X^{\delta,\frac{1}{2}+\delta}}.
}
\end{lemma}
Moreover, due to the Sobolev embedding we have the following result proved by Herr in \cite{h2006}.
\begin{lemma}[\cite{h2006}]
\label{bourgain-sobolev}
If $2\leq p,q<\infty$, $b\geq\frac{1}{2}-\frac{1}{p}$ and $s\geq \frac{1}{2}-\frac{1}{q}$, then
\eee{\label{bourgain-sobolev-est}
\no{u}_{L^p_t L^q_x}
\lesssim
\no{u}_{X^{s,b}}.
}
\end{lemma}
We now state and prove the nonlinear estimates needed for establishing Theorem \ref{dbp-wp}.
\begin{lemma}
\label{norm-form-1}
For any $\delta>0$, let $z\in H^s$ with $s>\frac{1}{2}+\delta$. Then, the quantity $NF(z)$ defined by \eqref{NF-def} satisfies the estimate
\ddd{\label{norm-form-1-est}
\no{NF(z)}_{H^{s+a}(\mathbb{T})}
\lesssim
\no{z}_{H^{\frac{1}{2}+\delta}(\mathbb{T})}^2
\no{z}_{H^s(\mathbb{T})}, \quad 0<a<1.
}
\end{lemma}
\begin{proof}
We have
\ddd{
\no{NF(z)}_{H^{s+a}(\mathbb{T})}
&=
\bigg(
\sum_{\xi} \langle{\xi\rangle}^{2(s+a)}
 \left| \widehat{NF(z)}(\xi)\right|^2
\bigg)^{\frac{1}{2}}
\nn\\
&\leq
\Bigg(
\sum_{\xi} \langle{\xi\rangle}^{2(s+a)}
\bigg(
\sum_{\substack{S_{\xi} \\ |\xi_1| \ll |\xi_2|}}
\frac{\langle{\xi_2\rangle}}{\langle{\xi_2-\xi_1\rangle}\langle{\xi_2-\xi_3\rangle}}
|\widehat{z}(\xi_1)| |{\widehat{z}(\xi_2)}| |\widehat{z}(\xi_3)|\bigg)^2
\Bigg)^{\frac{1}{2}}.
\nn
}
Next, we note that $|\xi-\xi_1|\sim |\xi-\xi_3|\sim |\xi_2|\sim |\xi|$. Applying this result and the Cauchy-Schwarz inequality in $\xi_1$ and $\xi_3$, we find
\ddd{
&\quad \ 
\Bigg(
\sum_{\xi} \langle{\xi\rangle}^{2(s+a)}
\bigg(
\sum_{\substack{\xi_1, \xi_3 \\ |\xi_1| \ll |\xi-\xi_1-\xi_3|}} 
\frac{\langle{\xi-\xi_1-\xi_3\rangle}}{\langle{\xi-\xi_1\rangle}\langle{\xi-\xi_3\rangle}} 
|\widehat{z}(\xi_1)| |{\widehat{z}(\xi-\xi_1-\xi_3)}| |\widehat{z}(\xi_3)|
\bigg)^2
\Bigg)^{\frac{1}{2}}
\nn\\
&\lesssim
\Bigg(
\sum_{\xi} \langle{\xi\rangle}^{2(s+a-1)}
\bigg(
\sum_{\xi_1, \xi_3} 
|\widehat{z}(\xi_1)| |{\widehat{z}(\xi-\xi_1-\xi_3)}| |\widehat{z}(\xi_3)|
\bigg)^2
\Bigg)^{\frac{1}{2}}
\nn
}
% @@ artificial breaking
\ddd{
&\lesssim
\Bigg(\sum_{\xi} \langle{\xi\rangle}^{2(a-1)} \sum_{\xi_1, \xi_3}
\langle{\xi_1\rangle}^{2(\frac{1}{2}+\delta)} |\widehat{z}(\xi_1)|^2
\langle{\xi-\xi_1-\xi_3\rangle}^{2s} |\widehat{z}(\xi-\xi_1-\xi_3)|^2
\langle{\xi_3\rangle}^{2(\frac{1}{2}+\delta)} |\widehat{z}(\xi_3)|^2
\Bigg)^{\frac{1}{2}}
\nn\\
&\leq
\sup_{\xi} \langle{\xi\rangle}^{(a-1)}
\no{z}_{H^{\frac{1}{2}+\delta}(\mathbb{T})}^2
\no{z}_{H^s(\mathbb{T})}.
\nn
}
The above supremum is finite for $a<1$, completing the proof.
\end{proof}
\begin{corollary}
\label{norm-form-2}
For any $\delta>0$, let $z\in C([0,T];H^s(\mathbb{T}))$ with $s>\frac{1}{2}+\delta$, $T>0$. Then,
\ddd{\label{norm-form-2-est}
\no{\int_0^t e^{i(t-t')\partial_x^2} \no{z(t')}_{L^4(\mathbb{T})}^4 NF(z)(t')~dt'}_{H^{s+a}(\mathbb{T})}
&\lesssim
T
\no{z}_{C([0,T];H^{\frac{1}{2}+\delta}(\mathbb{T}))}^6
\no{z}_{C([0,T];H^s(\mathbb{T}))}
}
for all $0<a<1$ and $t\in[0,T]$.
\end{corollary}
\begin{proof}
We simply note that
\ddd{
\no{\int_0^t e^{i(t-t')\partial_x^2} \no{z(t')}_{L^4(\mathbb{T})}^4 NF(z)(t')~dt'}_{H^{s+a}(\mathbb{T})}
&\leq 
\int_0^t \no{z(t')}_{L^4(\mathbb{T})}^4 \no{NF(z)(t')}_{H^{s+a}(\mathbb{T})}~dt'
\nn\\
&\leq 
T
\sup_{t\in[0,T]}
\no{z(t)}_{L^4(\mathbb{T})}^4
\no{NF(z)(t)}_{H^{s+a}(\mathbb{T})}.
\nn
}
The claim then follows from Sobolev embedding and Lemma \ref{norm-form-1}.
\end{proof}
\begin{lemma}
\label{non-dbp-terms}
Let $s>\frac{1}{2}$. Then, for all $\delta>0$ and $0<a<\min\left\{s-\frac{1}{2}-15\delta,\frac{1}{2}-15\delta\right\}$, the quantities $\mathcal B_\ell(z)$ defined by \eqref{B-def} satisfy the estimate
\eee{\label{non-dbp-terms-est}
\Big\| \sum_{\ell =1}^3 \mathcal{B}_{\ell}(z) \Big\|_{X^{s+a,-\frac{1}{2}+\delta}}
\lesssim
\no{z}_{Z^{\sigma}}^2
\no{z}_{Z^s}.
}
\end{lemma}

\begin{proof}
First, recall that $\max\left\{\left|\tau_m +\xi_m^2\right|\right\}_{m=0}^3 \gtrsim |(\xi_2-\xi_1)(\xi_2-\xi_3)|$. By applying Littlewood-Paley-type projections, Plancherel's theorem and duality, we have 
\ddd{
&\quad
\no{\mathcal{B}_1(z)}_{X^{s+a,-\frac{1}{2}+\delta}}
\nn\\
&\lesssim
\sum_{j,k} 2^{j(s+a)} 2^k 
\sup_{\no{\varphi}_{X^{0,\frac{1}{2}-\delta}}=1} 
\sum_{\substack{S_{\xi} \\ |\xi_1| \gg |\xi_2|}}  \underset{\tau_1-\tau_2+\tau_3 = \tau}{\int}  
|\widetilde{z}_{\sim j}(\xi_1,\tau_1)|
|\widetilde{z}_k(\xi_2,\tau_2)|
|\widetilde{z}(\xi_3,\tau_3)|
|\widetilde{\varphi_j}(\xi,\tau)|
~d\sigma.
\nn
}

First, suppose that $\max\left\{\left|\tau_m +\xi_m^2\right|\right\}_{m=0}^3=|\tau+\xi^2|$. Then,
\ddd{
&\quad 
\sum_{\substack{S_{\xi} \\ |\xi_1| \gg |\xi_2|}} \underset{\tau_1-\tau_2+\tau_3 = \tau}{\int}
|\widetilde{z}_{\sim j}(\xi_1,\tau_1)|
|\widetilde{z}_k(\xi_2,\tau_2)|
|\widetilde{z}(\xi_3,\tau_3)|
|\widetilde{\varphi_j}(\xi,\tau)|
~d\sigma
\nn\\
&\lesssim
\sum_{\substack{S_{\xi} \\ |\xi_1| \gg |\xi_2|}} \underset{\tau_1-\tau_2+\tau_3 = \tau}{\int}
\frac{\langle{\tau+\xi^2\rangle}^{\frac{1}{2}-\delta}}{\langle{\xi_2-\xi_1\rangle}^{\frac{1}{2}-\delta} \langle{\xi_2-\xi_3\rangle}^{\frac{1}{2}-\delta} }
|\widetilde{z}_{\sim j}(\xi_1,\tau_1)|
|\widetilde{z}_k(\xi_2,\tau_2)|
|\widetilde{z}(\xi_3,\tau_3)|
|\widetilde{\varphi_j}(\xi,\tau)|
~d\sigma,
\nn
}
where $d\sigma$ is the surface measure.
Since $|\xi_1|\gg |\xi_2|$, we may apply H\"older's inequality and estimate \eqref{l4-est} to arrive at
\ddd{
&\quad
\sum_{j,k} 2^{j(s+a)} 2^k \sup_{\no{\varphi}_{X^{0,\frac{1}{2}-\delta}}=1}
\sum_{\substack{S_{\xi} \\ |\xi_1| \gg |\xi_2|}} \underset{\tau_1-\tau_2+\tau_3 = \tau}{\int}
|\widetilde{z}_{\sim j}(\xi_1,\tau_1)|
|\widetilde{z}_k(\xi_2,\tau_2)|
|\widetilde{z}(\xi_3,\tau_3)|
|\widetilde{\varphi_j}(\xi,\tau)|
~d\sigma
\nn\\
&\lesssim
\sum_{j,k} 2^{j(s+a-\frac{1}{2}+\delta)} 2^k
\big\|(z_{\sim j})^{\#} (z_k)^{\#} z^{\#}\big\|_{L^2_{t,x}}
\nn\\
&\lesssim
\sum_{j,k} 2^{j(s+a-\frac{1}{2}+\delta)} 2^k 
\big\|z^{\#}\big\|_{L^{\infty}_{t,x}} 
\big\|(z_{\sim j})^{\#}\big\|_{L^4_{t,x}}
\big\|(z_k)^{\#}\big\|_{L^4_{t,x}}
\nn\\
&\lesssim
\sum_{j,k} 2^{j(a-\frac{1}{2}+\delta)} 2^{k(1-\sigma)}
\no{z}_{Y^{\frac{1}{2}+\delta,0}}
\no{z}_{X^{\sigma,\frac{1}{2}}}
\no{z}_{X^{s,\frac{1}{2}}}.
\nn
}
If $s<1$, then
\ddd{
\sum_{j,k} 2^{j(a-\frac{1}{2}+\delta)} 2^{k(1-\sigma)}
&\lesssim
\sum_j 2^{j(a+\frac{1}{2}+\delta-s)},
\nn
}
which converges provided that $a< s-\frac{1}{2}-\delta$. 
If $s\geq1$, then the sum converges for $a<\frac{1}{2}-\delta$. If $\max\left\{\left|\tau_m +\xi_m^2\right|\right\}_{m=0}^3=|\tau_1+\xi_1^2|$ or $\max\left\{|\tau_m +\xi_m^2|\right\}_{m=0}^3=|\tau_2+\xi_2^2|$, a similar argument yields the same result. 
\\
\indent
Next, recall that the interpolation of the Bourgain spaces $X^{s_1, b_1}$ and $X^{s_2, b_2}$ at $\theta$ is given by 
$$
[X^{s_1,b_1}, X^{s_2, b_2}]_{\theta}
=
X^{s_1(1-\theta)+s_2\theta, b_1(1-\theta)+b_2\theta}.
$$
By estimates \eqref{l6-est} and \eqref{bourgain-sobolev-est}, we have $X^{\delta,\frac{1}{2}+\delta}\hookrightarrow L^6(\mathbb{T}\times\mathbb{R})$ and $X^{\frac{1}{3},\frac{1}{3}}\hookrightarrow L^6(\mathbb{T}\times\mathbb{R})$, respectively. Thus,
$$
[X^{\frac{1}{3},\frac{1}{3}}, X^{\delta,\frac{1}{2}+\delta}]_{\theta}
\hookrightarrow
[L^6(\mathbb{T}\times\mathbb{R}), L^6(\mathbb{T}\times\mathbb{R})]_{\theta}
=
L^6(\mathbb{T}\times\mathbb{R})
$$
which for $\theta = (\frac{1}{3}-\delta)/(\frac{1}{6}+\delta)$ gives 
$$
X^{5\delta,\frac{1}{2}-\delta}
\hookrightarrow
[X^{\frac{1}{3},\frac{1}{3}}, X^{\delta,\frac{1}{2}+\delta}]_{\theta}
\hookrightarrow
L^6(\mathbb{T}\times\mathbb{R}).
$$
If $\max\left\{|\tau_m +\xi_m^2|\right\}_{m=0}^3=|\tau_3+\xi_3^2|$, then we apply $X^{5\delta, \frac{1}{2}-\delta}\hookrightarrow L^6(\mathbb{T}\times\mathbb{R})$ along with H\"older's inequality to yield
\ddd{
&\sum_{j,k} 2^{j(s+a)} 2^k \sup_{\no{\varphi}_{X^{0,\frac{1}{2}-\delta}}=1}
\sum_{\substack{S_{\xi} \\ |\xi_1| \gg |\xi_2|}} \underset{\tau_1-\tau_2+\tau_3 = \tau}{\int}
|\widetilde{z}_{\sim j}(\xi_1,\tau_1)|
|\widetilde{z}_k(\xi_2,\tau_2)|
|\widetilde{z}(\xi_3,\tau_3)|
|\widetilde{\varphi_j}(\xi,\tau)|
~d\sigma
\nn\\
&\lesssim
\sum_{j,k} 2^{j(s+a-\frac{1}{2})} 2^k \no{z}_{X^{0,\frac{1}{2}}} 
\sup_{\no{\varphi}_{X^{0,\frac{1}{2}-\delta}}=1}
\big\|(z_{\sim j})^{\#} (z_k)^{\#} (\varphi_j)^{\#}\big\|_{L^2_{t,x}}
\nn\\
&\lesssim
\sum_{j,k} 2^{j(a+10\delta-\frac{1}{2})} 2^{k(1+5\delta-\sigma)} 
\no{z}_{X^{\sigma,\frac{1}{2}}}^2
\no{z}_{X^{s,\frac{1}{2}}}
\nn
}
If $s<1$, then
\ddd{
\sum_{j,k} 2^{j(a+10\delta-\frac{1}{2})} 2^{k(1+5\delta-s)}
&\lesssim
\sum_{j,k} 2^{j(a+15\delta+\frac{1}{2}-s)},
\nn
}
which converges for $a<s-\frac{1}{2}-15\delta$. If $s\geq 1$, then the sum converges for $a<\frac{1}{2}-15\delta$.
Overall, we have shown that $\mathcal{B}_1(z) \in X^{s+a,-\frac{1}{2}+\delta}$ provided that $0<a<\min\left\{s-\frac{1}{2}-15\delta, \frac{1}{2}-15\delta\right\}$.

Next, we note that by the definition \eqref{B-def} in the case of $\widehat{\mathcal{B}_2(z)}(\xi)$ we have $|\xi_1|\sim |\xi_2|\gg |\xi_3|$. Thus, $|\xi_2-\xi_3|\sim |\xi_2|$, which allows us to yield the same estimates as those obtained for $\mathcal{B}_1(z)$ and thereby deduce that $\mathcal{B}_2(z)\in X^{s+a,-\frac{1}{2}+\delta}$ for $a<\min\left\{s-\frac{1}{2}-15\delta, \frac{1}{2}-15\delta\right\}$.

Finally, by applying Littlewood-Paley-type projections, H\"older's inequality and estimate \eqref{l4-est} once more, we have
\ddd{
\no{\mathcal{B}_3(z)}_{X^{s+a,-\frac{1}{2}+\delta}}
&\lesssim
\sum_j 2^{j(s+a+1)}\sup_{\no{\varphi}_{X^{0,\frac{1}{2}-\delta}}}
\int_{\mathbb{T}\times\mathbb{R}}
[(z_{\sim j})^{\#}]^3 (\varphi_j)^{\#}
~dxdt
\nn\\
&\lesssim
\sum_j 2^{j(s+a+1)}\sup_{\no{\varphi}_{X^{0,\frac{1}{2}-\delta}}}
\big\|(z_{\sim j})^{\#}\big\|_{L^4_{t,x}}^3 \big\|(\varphi_j)^{\#}\big\|_{L^4_{t,x}}
\nn\\
&\lesssim
\sum_j 2^{j(a+1-2\sigma)} 
\no{z}_{X^{\sigma,\frac{1}{2}}}^2
\no{z}_{X^{s,\frac{1}{2}}},
\nn
}
where the above sum converges for $a<2\sigma-1=\min\left\{2s-1, 1\right\}$.
\end{proof}

\begin{lemma}
\label{nonres-quintic}
Let $s>\frac{1}{2}$. Then, the quantity $\mathcal A(z)$ defined analogously to \eqref{A-def} satisfies the estimate
\eee{
\label{nonres-quintic-est}
\no{\mathcal{A}(z)}_{X^{s+a,-\frac{1}{2}+\delta}}
\lesssim
\no{z}_{Z^{\sigma}}^4
\no{z}_{Z^s}
}
for all $\delta>0$ and $0<a<\frac{1}{2}-\delta$.
\end{lemma}
\begin{proof}
We first apply Littlewood-Paley projections, duality and Plancherel's theorem to infer
\ddd{
\no{\mathcal{A}(z)}_{X^{s+a,-\frac{1}{2}+\delta}}
&\lesssim
\sum_j 2^{j(s+a)}
\sup_{\no{\varphi}_{X^{0,\frac{1}{2}-\delta}}=1}
\sum_{\ell=1}^4 R_{\ell},
\nn
}
where
\eee{
R_1
=
\sum_{\substack{A(\xi) \\ |\xi_1^*|\gg |\xi_2^*|^2}}
\underset{\tau_1-\tau_2+\tau_3-\tau_4+\tau_5=\tau}{\int}
|\widetilde{z}(\xi_1,\tau_1)|
|\widetilde{z}(\xi_2,\tau_2)|
|\widetilde{z}(\xi_3,\tau_3)|
|\widetilde{z}(\xi_4,\tau_4)|
|\widetilde{z}(\xi_5,\tau_5)|
|\widetilde{\varphi_j}(\xi,\tau)|
~d\sigma,
\nn
}
\eee{
R_2
=
\sum_{\substack{A(\xi) \\ |\xi_3^*|^2 \ll |\xi_1^*|\lesssim |\xi_2^*|^2 \\  |\xi_1^*| \gg |\xi_2^*|}}
\underset{\tau_1-\tau_2+\tau_3-\tau_4+\tau_5=\tau}{\int}
|\widetilde{z}(\xi_1,\tau_1)|
|\widetilde{z}(\xi_2,\tau_2)|
|\widetilde{z}(\xi_3,\tau_3)|
|\widetilde{z}(\xi_4,\tau_4)|
|\widetilde{z}(\xi_5,\tau_5)|
|\widetilde{\varphi_j}(\xi,\tau)|
~d\sigma,
\nn
}
\eee{
R_3
=
\sum_{\substack{A(\xi) \\ |\xi_1^*|\lesssim |\xi_2^*|^2, |\xi_3^*|^2 \\ |\xi_1^*| \gg |\xi_2^*|}}
\underset{\tau_1-\tau_2+\tau_3-\tau_4+\tau_5=\tau}{\int}
|\widetilde{z}(\xi_1,\tau_1)|
|\widetilde{z}(\xi_2,\tau_2)|
|\widetilde{z}(\xi_3,\tau_3)|
|\widetilde{z}(\xi_4,\tau_4)|
|\widetilde{z}(\xi_5,\tau_5)|
|\widetilde{\varphi_j}(\xi,\tau)|
~d\sigma
\nn
}
\eee{
R_4
=
\sum_{\substack{A(\xi) \\ |\xi_1^*|\lesssim |\xi_2^*|}}
\underset{\tau_1-\tau_2+\tau_3-\tau_4+\tau_5=\tau}{\int}
|\widetilde{z}(\xi_1,\tau_1)|
|\widetilde{z}(\xi_2,\tau_2)|
|\widetilde{z}(\xi_3,\tau_3)|
|\widetilde{z}(\xi_4,\tau_4)|
|\widetilde{z}(\xi_5,\tau_5)|
|\widetilde{\varphi_j}(\xi,\tau)|
~d\sigma,
\nn
}
and  $\xi_j^*$ denotes the $j$th  frequency among $\xi_1, \ldots, \xi_5$ after ordering all five frequencies with respect to the size of their absolute value, i.e.   $|\xi_1^*|\geq |\xi_2^*|\geq |\xi_3^*| \geq |\xi_4^*| \geq |\xi_5^*|$.
Next, we note that in  the above range of summation and integration we have
$\max\left\{|\tau_{\ell}+\xi_{\ell}^2|\right\}_{\ell=0}^5
\gtrsim
|\Phi|
$,
where $(\xi_0,\tau_0)=(\xi,\tau)$ and
\eee{
\Phi
=
\xi_2^2+\xi_4^2+\xi_2\xi_4+\xi_1(-\xi_2+\xi_3-\xi_4+\xi_5)-\xi_2(\xi_3+\xi_5)+\xi_3(-\xi_4+\xi_5)-\xi_4\xi_5.\nn
}

First, we consider the case $|\xi_1^*|\gg |\xi_2^*|^2$ which corresponds to $R_1$. Suppose that $\xi_1=\xi_1^*$. Since $\xi_2+\xi_4\neq \xi_3+\xi_5$, we have  
\eee{
|\Phi|
\geq
|\xi_1|-8|\xi_2^*|^2
\sim 
|\xi_1|.
\nn
}
The same argument holds if $\xi_3=\xi_1^*$ or $\xi_5=\xi_1^*$.
Next, suppose that $\xi_2=\xi_1^*$. Then
\eee{
|\Phi|
\geq
|\xi_2| |\xi_1-\xi_2+\xi_3-\xi_4+\xi_5|
-
7|\xi_2^*|^2
\geq
|\xi_2|^2-7|\xi_2^*|^2
\sim
|\xi_2|^2
=
|\xi_1^*|^2
\nn
}
while the same argument holds if $\xi_4=\xi_1^*$.

Second, we consider the case $\{|\xi_3^*|^2 \ll |\xi_1^*| \lesssim |\xi_2^*|^2\}\cap \{|\xi_1^*|\gg |\xi_2^*|\}$ which corresponds to $R_2$. If $\xi_1=\xi_1^*$, then
\ddd{
|\Phi|
&\geq
|\xi_1| |-\xi_2+\xi_3-\xi_4+\xi_5| - 8|\xi_2^*|^2
\nn\\
&\gtrsim
|\xi_1| |\xi_2^*|
-8|\xi_2^*|^2
=
|\xi_2^*| (|\xi_1|-8|\xi_2^*|)
\sim
|\xi_1| |\xi_2^*| = |\xi_1^*|^2.
\nn
}
The same result is valid if $\xi_3=\xi_1^*$ or $\xi_5=\xi_1^*$. Moreover, if $\xi_2=\xi_1^*$ or $\xi_4=\xi_1^*$ then  the same argument as before yields
\ddd{
|\Phi|
&\gtrsim
|\xi_1^*|^2.
\nn
}

Thus, we have now shown that within the range of summation and integration associated to $R_1$ and $R_2$ we have $|\Phi|\gtrsim |\xi_1^*|^2$. Therefore, by an application of H\"older's inequality, the Sobolev embedding  and estimate \eqref{l4-est}, we obtain
\ddd{
\sum_j 2^{j(s+a)}
\sup_{\no{\varphi}_{X^{0,\frac{1}{2}-\delta}}=1}
(R_1+R_2)
&\lesssim
\sum_j 2^{j(a-\frac{1}{2}+\delta)}
\no{z}_{Y^{\frac{1}{2}+\delta,0}}^3
\no{z}_{X^{0,\frac{1}{2}}}
\no{z}_{X^{s,\frac{1}{2}}},
\nn
}
where the  sum on the right-hand side converges for $a<\frac{1}{2}-\delta$.

In the region of summation of $R_3$  we have  $|\xi_1^*|^{\frac{1}{2}}\sim |\xi_2^*|\sim |\xi_3^*|$, while in the region of summation of $R_4$  we have $|\xi_1^*|\sim |\xi_2^*|$. Thus, using once again H\"older's inequality and the Sobolev embedding along with estimate \eqref{l4-est}, we find
\ddd{
\sum_j 2^{j(s+a)}
\sup_{\no{\varphi}_{X^{0,\frac{1}{2}-\delta}}=1}
(R_3+R_4)
&\lesssim
\sum_j 2^{j(a-\sigma)}
\no{z}_{Y^{\frac{1}{2}+\delta,0}}^2
\no{z}_{X^{\sigma,\frac{1}{2}}}^2
\no{z}_{X^{s,\frac{1}{2}}},
}
where the sum on the right-hand side converges provided that $a<\sigma=\min\left\{s,1\right\}$.
\end{proof}
\begin{corollary}
\label{nonres-cubic}
Let $s>\frac{1}{2}$. Then, the quantity $NR(|z|^2z)$ defined analogously to \eqref{NR-def} satisfies 
\eee{
\label{nonres-cubic-est}
\no{NR(|z|^2 z)}_{X^{s+a,-\frac{1}{2}+\delta}}
\lesssim
\no{z}_{Z^{\sigma}}^2
\no{z}_{Z^s}
}
for all $\delta>0$ and $0<a<\min\{s-\frac{1}{2}-15\delta, \frac{1}{2}-15\delta\}$.
\end{corollary}
%\begin{remark}
%Clearly, \eqref{nonres-cubic-est} holds for larger $a>0$ than specified in Corollary \ref{nonres-cubic}.
%\end{remark}
\begin{lemma}
\label{n11-n13}
Let $s>\frac{1}{2}$. Then, $\mathcal N_{1, 1}(z)$ and $\mathcal N_{1, 3}(z)$ defined by \eqref{N1l-def} satisfy the estimate
\eee{\label{n11-n13-est}
\no{\mathcal{N}_{1,1}(z)+\mathcal{N}_{1,3}(z)}_{X^{s+a,-\frac{1}{2}+\delta}}
\lesssim
\no{z}_{Z^{\sigma}}^4
\no{z}_{Z^s}
}
for all $\delta>0$ and $0<a<\min\{s,1\}$.
\end{lemma}
\begin{proof}
We only prove the estimate for $\mathcal{N}_{1,1}(z)$ as the estimate for $\mathcal{N}_{1,3}(z)$ can be established similarly. Note that the range of summation of $\widehat{\mathcal{N}_{1,1}(z)}(\xi)$ implies $|\xi_5|\leq |\xi_1-\xi_2+\xi_3|\ll |\xi_4|$. Thus, $|\Psi|\gtrsim |\xi_4|^2\sim |\xi|^2$. Then, using  Littlewood-Paley-type projections, duality and Plancherel's theorem, we find
\eee{
\no{\mathcal{N}_{1,1}(z)}_{X^{s+a,-\frac{1}{2}+\delta}}
\lesssim
\sum_{j,k}
2^{j(s+a-1)} 2^k
\sup_{\no{\varphi}_{X^{0,\frac{1}{2}-\delta}}=1}
\int_{\mathbb{T}\times\mathbb{R}}
(z^{\#})^3 (z_{\sim j})^{\#} (z_k)^{\#} (\varphi_j)^{\#}
~dxdt.
\nn
}

First, suppose that $|\xi_2| \ll |\xi_4|$.
From H\"older's inequality and estimate \eqref{l4-est}, we have %
\ddd{
&\quad \sum_{j,k}
2^{j(s+a-1)} 2^k
\sup_{\no{\varphi}_{X^{0,\frac{1}{2}-\delta}}=1}
\int_{\mathbb{T}\times\mathbb{R}}
(z^{\#})^3 (z_{\sim j})^{\#} (z_k)^{\#} (\varphi_j)^{\#}
~dxdt
\nn\\
&\lesssim
\sum_{j,k}
2^{j(s+a-1)} 2^k
\big\|z^{\#}\big\|_{L^{\infty}_{t,x}}^3 
\big\|(z_{\sim j})^{\#}\big\|_{L^4_{t,x}}
\big\|(z_k)^{\#}\big\|_{L^4_{t,x}}
\nn\\
&\lesssim
\sum_{j,k} 2^{j(a-1)} 2^{k(1-\sigma)}
\no{z}_{Y^{\frac{1}{2}+\delta,0}}^3
\no{z}_{X^{s,\frac{1}{2}}}
\no{z}_{X^{\sigma,\frac{1}{2}}}.
\nn
}
If $s<1$, then 
\eee{
\sum_{j,k} 2^{j(a-1)} 2^{k(1-s)}
\lesssim
\sum_j 2^{j(a-s)},
\nn
}
which converges for $a<s$. On the other hand, if $s\geq1$  then the above sum converges for $a<1$. Thus, overall we require $a<\min\{s, 1\}$.
 
If $|\xi_2|\gtrsim |\xi_4|$, then it follows that $|\xi_2|\sim |\xi_1+\xi_3| \lesssim \max\{|\xi_1|, |\xi_3|\}$. Therefore,
\ddd{
&\quad \sum_{j,k}
2^{j(s+a-1)} 2^k
\sup_{\no{\varphi}_{X^{0,\frac{1}{2}-\delta}}=1}
\int_{\mathbb{T}\times\mathbb{R}}
(z^{\#})^2 (z_{\sim k})^{\#} (z_{\sim j})^{\#} (z_k)^{\#} (\varphi_j)^{\#}
~dxdt
\nn\\
&\lesssim
\sum_{j,k} 2^{j(a-1)} 2^{k(1-2\sigma)} \no{z}_{Y^{\frac{1}{2}+\delta,0}}^2 \no{z}_{X^{\sigma,\frac{1}{2}}}^2 \no{z}_{X^{s,\frac{1}{2}}}
\nn\\
&\lesssim
\sum_j 2^{j(a-2\sigma)} 
\no{z}_{Y^{\frac{1}{2}+\delta,0}}^2 \no{z}_{X^{\sigma,\frac{1}{2}}}^2 
\no{z}_{X^{s,\frac{1}{2}}},
\nn
}
where the above sum converges for $a<2\sigma=\min\{2s, 2\}$.
\end{proof}
\begin{corollary}
\label{n23}
Let $s>\frac{1}{2}$. Then, $\mathcal N_{2, 3}(z)$ defined by \eqref{N2l-def} satisfies the estimate
\eee{\label{n23-est}
\no{\mathcal{N}_{2,3}(z)}_{X^{s+a,-\frac{1}{2}+\delta}}
\lesssim
\no{z}_{Z^{\sigma}}^4
\no{z}_{Z^s}
}
for all $\delta >0$ and $0<a<\min\{s,1\}$.
\end{corollary}
\begin{lemma}
\label{n12}
Let $s>\frac{1}{2}$. Then, $\mathcal N_{1, 2}(z)$ defined by \eqref{N1l-def} satisfies the estimate 
\eee{\label{n12-est}
\no{\mathcal{N}_{1,2}(z)}_{X^{s+a,-\frac{1}{2}+\delta}}
\lesssim
\no{z}_{Y^{\frac{1}{2}+\delta,0}}^6 \no{z}_{X^{s,\frac{1}{2}}}
}
for all $\delta>0$ and $0<a<1$.
\end{lemma}
\begin{proof}
As in the proof of Lemma \ref{n11-n13}, we have that $|\Psi|\gtrsim |\xi_6|^2\sim |\xi|^2$ and so applying Littlewood-Paley-type projections, duality, Plancherel's theorem  and H\"older's inequality we obtain
\ddd{
\no{\mathcal{N}_{1,2}(z)}_{X^{s+a,-\frac{1}{2}+\delta}}
&\lesssim
\sum_j 2^{j(s+a-1)} 
\sup_{\no{\varphi}_{X^{0,\frac{1}{2}-\delta}}=1}
\int_{\mathbb{T}\times\mathbb{R}}
(z^{\#})^6 (z_{\sim j})^{\#} (\varphi_j)^{\#}
~dxdt
\nn\\
&\lesssim
\sum_j 2^{j(a-1)}
\no{z}_{Y^{\frac{1}{2}+\delta,0}}^6 \no{z}_{X^{s,\frac{1}{2}}},
\nn
}
where the sum converges for $a<1$.
\end{proof}
\begin{corollary}
\label{n22}
Let $s>\frac{1}{2}$. Then, $\mathcal N_{2, 2}(z)$ defined by \eqref{N2l-def} satisfies the estimate 
\eee{\label{n22-est}
\no{\mathcal{N}_{2,2}(z)}_{X^{s+a,-\frac{1}{2}+\delta}}
\lesssim
\no{z}_{Y^{\frac{1}{2}+\delta,0}}^6 \no{z}_{X^{s,\frac{1}{2}}}
}
for all $\delta>0$ and $0<a<1$.
\end{corollary}
\begin{lemma}
\label{n21-star}
Let $s>\frac{1}{2}$. Then, $\mathcal N_{2, 1}^*(z)$ defined by \eqref{N2l*-def} satisfies the estimate 
\ddd{
\label{n21-star-est}
\no{\mathcal{N}_{2,1}^*(z)}_{X^{s+a,-\frac{1}{2}+\delta}}
&\lesssim
\no{z}_{Z^{\sigma}}^4
\no{z}_{Z^s}
}
for all $\delta>0$ and $0<a<\frac{1}{2}$.
\end{lemma}

\begin{proof}
Employing Littlewood-Paley-type projections, duality  and Plancherel's theorem, we find
\ddd{
\no{\mathcal{N}_{2,1}^*(z)}_{X^{s+a,-\frac{1}{2}+\delta}}
&\lesssim
\sum_{j,k} 2^{j(s+a-1)} 2^k 
\sup_{\no{\varphi}_{X^{0,\frac{1}{2}-\delta}}=1}
\sum_{\ell=1}^4
R_{\ell},
\nn}
where
\eee{
R_1
=
\sum_{\substack{N_{2,1}(\xi) \\ \xi_2+\xi_4\neq \xi_1+\xi_5 \\ |\xi_3|\gg |\xi_2^*|^2}}
\underset{\tau_1-\tau_2+\tau_3-\tau_4+\tau_5=\tau}{\int}
|\widetilde{z}(\xi_1,\tau_1)|
|\widetilde{z}(\xi_2,\tau_2)|
|\widetilde{z_k}(\xi_3,\tau_3)|
|\widetilde{z}(\xi_4,\tau_4)|
|\widetilde{z}(\xi_5,\tau_5)|
|\varphi_j(\xi,\tau)|
~d\sigma,
\nn
}
\eee{
R_2
=
\sum_{\substack{N_{2,1}(\xi) \\ \xi_2+\xi_4\neq \xi_1+\xi_5 \\ |\xi_3^*|^2\ll |\xi_3|\lesssim |\xi_2^*|^2 \\ |\xi_3|\gg |\xi_2^*|}}
\underset{\tau_1-\tau_2+\tau_3-\tau_4+\tau_5=\tau}{\int}
|\widetilde{z}(\xi_1,\tau_1)|
|\widetilde{z}(\xi_2,\tau_2)|
|\widetilde{z_k}(\xi_3,\tau_3)|
|\widetilde{z}(\xi_4,\tau_4)|
|\widetilde{z}(\xi_5,\tau_5)|
|\varphi_j(\xi,\tau)|
~d\sigma,
\nn
}
\eee{
R_3
=
\sum_{\substack{N_{2,1}(\xi) \\ \xi_2+\xi_4\neq \xi_1+\xi_5 \\ |\xi_3|\sim |\xi_2^*|^2, |\xi_3^*|^2 \\ |\xi_3|\gg |\xi_2^*|}}
\underset{\tau_1-\tau_2+\tau_3-\tau_4+\tau_5=\tau}{\int}
|\widetilde{z}(\xi_1,\tau_1)|
|\widetilde{z}(\xi_2,\tau_2)|
|\widetilde{z_k}(\xi_3,\tau_3)|
|\widetilde{z}(\xi_4,\tau_4)|
|\widetilde{z}(\xi_5,\tau_5)|
|\varphi_j(\xi,\tau)|
~d\sigma,
\nn
}
\eee{
R_4
=
\sum_{\substack{N_{2,1}(\xi) \\ \xi_2+\xi_4\neq \xi_1+\xi_5 \\ |\xi_3|\lesssim |\xi_2^*|}}
\underset{\tau_1-\tau_2+\tau_3-\tau_4+\tau_5=\tau}{\int}
|\widetilde{z}(\xi_1,\tau_1)|
|\widetilde{z}(\xi_2,\tau_2)|
|\widetilde{z_k}(\xi_3,\tau_3)|
|\widetilde{z}(\xi_4,\tau_4)|
|\widetilde{z}(\xi_5,\tau_5)|
|\varphi_j(\xi,\tau)|
~d\sigma.
\nn
}
We recall that in the case of $R_1$ and $R_2$ we have $\max\left\{\left|\tau_m +\xi_m^2\right|\right\}_{m=0}^5 \gtrsim |\xi_3|$. Therefore,   H\"older's inequality, the Sobolev embedding  and estimate \eqref{l4-est} yield
\ddd{
\sum_{j,k} 2^{j(s+a-1)} 2^k 
\sup_{\no{\varphi}_{X^{0,\frac{1}{2}-\delta}}=1}
(R_1+R_2)
&\lesssim
\sum_{j,k} 2^{j(s+a-1)} 2^{k(\frac{1}{2}-s)}
\no{z}_{Z^{\sigma}}^4
\no{z}_{Z^s}
\nn\\
&\lesssim
\sum_j 2^{j(a-\frac{1}{2})} 
\no{z}_{Z^{\sigma}}^4
\no{z}_{Z^s},
\nn
}
where the above sum is finite provided that $a<\frac{1}{2}$.

Regarding $R_3$, we note that there exist $\ell, m\neq 3$ so that $|\xi_{\ell}|^2 \sim |\xi_m|^2 \sim |\xi_3|$. Finally, in the case of $R_4$  there exists $\ell\neq 3$ such that $|\xi_{\ell}|\gtrsim 2^k$. Employing once again  H\"older's inequality, the Sobolev embedding and estimate \eqref{l4-est}, we deduce
\ddd{
\sum_{j,k} 2^{j(s+a-1)} 2^k 
\sup_{\no{\varphi}_{X^{0,\frac{1}{2}-\delta}}=1}
(R_3+R_4)
&\lesssim
\sum_{j,k}
2^{j(s+a-1)} 2^{k(1-s-\sigma)}
\no{z}_{Z^{\sigma}}^4
\no{z}_{Z^s}
\nn\\
&\lesssim
\sum_j 2^{j(a-\sigma)}
\no{z}_{Z^{\sigma}}^4
\no{z}_{Z^s} 
\nn
}
and the right-hand side is finite for $a<\sigma=\min\{s,1\}$.
\end{proof}
\begin{lemma}
\label{n22-star}
Let $s>\frac{1}{2}$. Then, $\mathcal N_{2, 2}^*(z)$ defined by \eqref{N2l*-def} satisfies the estimate 
\eee{
\label{n22-star-est}
\no{\mathcal{N}_{2,2}^*(z)}_{X^{s+a,-\frac{1}{2}+\delta}}
\lesssim
\no{z}_{Z^{\sigma}}^4
\no{z}_{Z^s}
}
for all $\delta>0$ and $0<a<\min\{2s, 2\}$.
\end{lemma}
\begin{proof}
First, note that $\Psi=(\xi-\xi_1)(\xi-\xi_5)=(\xi-(\xi_1+\xi_5)+\xi_5)(\xi-(\xi_1+\xi_5)+\xi_1)$. Since $|\xi-(\xi_1+\xi_5)|\gg |\xi_1|, |\xi_5|$, we have that $|\Psi|\gtrsim |\xi-(\xi_1+\xi_5)|^2\sim |\xi|^2$.
By applying Littlewood-Paley-type projections, duality  and Plancherel's theorem, we have
\eee{
\no{\mathcal{N}_{2,2}^*(z)}_{X^{s+a,-\frac{1}{2}+\delta}}
\lesssim
\sum_{j,k,\ell}
2^{j(s+a-2)} 2^k 2^{\ell}
\sup_{\no{\varphi}_{X^{0,\frac{1}{2}-\delta}}=1}
R^*,
\nn
}
where 
\eee{
R^*
=
\sum_{\substack{N_{2,1}(\xi) \\ \xi_2+\xi_4= \xi_1+\xi_3}}
\underset{\tau_1-\tau_2+\tau_3-\tau_4+\tau_5=\tau}{\int}
|\widetilde{z_k}(\xi_1,\tau_1)|
|\widetilde{z}(\xi_2,\tau_2)|
|\widetilde{z_{\ell}}(\xi_3,\tau_3)|
|\widetilde{z}(\xi_4,\tau_4)|
|\widetilde{z_j}(\xi,\tau_5)|
|\widetilde{\varphi_j}(\xi,\tau)|
~d\sigma.
\nn
}
Noting that
\ddd{
R^*
&\leq
\!
\sum_{\xi_1-\xi_2+\xi_3-\xi_4+\xi_5=\xi}
\underset{\tau_1-\tau_2+\tau_3-\tau_4+\tau_5=\tau}{\int}
\!\!\!\!\!
|\widetilde{z_k}(\xi_1,\tau_1)|
|\widetilde{z}(\xi_2,\tau_2)|
|\widetilde{z_{\ell}}(\xi_3,\tau_3)|
|\widetilde{z}(\xi_4,\tau_4)|
|\widetilde{z_j}(\xi_5,\tau_5)|
|\widetilde{\varphi_j}(\xi,\tau)|
 d\sigma
\nn\\
&\simeq
\int_{\mathbb{T}\times\mathbb{R}}
(z^{\#})^2 (z_j)^{\#} (z_k)^{\#} (z_{\ell})^{\#}
(\varphi_j)^{\#}
~dxdt
\nn
}
and applying H\"older's inequality, the Sobolev embedding and estimate \eqref{l4-est}, we obtain
\ddd{
\no{\mathcal{N}_{2,2}^*(z)}_{X^{s+a,-\frac{1}{2}+\delta}}
&\lesssim
\sum_{j,k,\ell}
2^{j(s+a-2)} 2^k 2^{\ell}
\sup_{\no{\varphi}_{X^{0,\frac{1}{2}-\delta}}=1}
\int_{\mathbb{T}\times\mathbb{R}}
(z^{\#})^2 (z_j)^{\#} (z_k)^{\#} (z_{\ell})^{\#}
(\varphi_j)^{\#}
~dxdt
\nn\\
&\lesssim
\sum_{j,k,\ell} 2^{j(a-2)} 2^{k(1-\sigma)} 2^{\ell(1-\sigma)}
\no{z}_{Y^{\frac{1}{2}+\delta,0}}^2 
\no{z}_{X^{\sigma,\frac{1}{2}}}^2
\no{z}_{X^{s,\frac{1}{2}}}.
\nn
}
If $s<1$, we have
\eee{
\sum_{j,k,\ell} 2^{j(a-2)} 2^{k(1-s)} 2^{\ell(1-s)}
\lesssim
\sum_j 2^{j(a-2s)},
\nn
}
which converges for $a<2s$. If $s\geq1$, the sum converges for $a<2$.
\end{proof}
\begin{lemma}\label{error}
Let $s>\frac{1}{2}$. Then, the quantities $\mathcal E_1(z)$ and $\mathcal E_2(z)$ given by \eqref{E-def} and \eqref{E-def2} satisfy 
\eee{\label{error-est}
\no{\mathcal{E}_1(z)+\mathcal{E}_2(z)}_{X^{s+a,-\frac{1}{2}+\delta}}
\lesssim
\no{z}_{Z^{\sigma}}^4
\no{z}_{Z^s}
}
for all $\delta>0$ and $0<a<\min\{s,1\}$.
\end{lemma}
\begin{proof}
We only show the estimate for $\mathcal{E}_1(z)$ as all of the terms in $\mathcal{E}_2(z)$ possess the same structure necessary to yield the same estimate for $\mathcal{E}_2(z)$. First, by applying Littlewood-Paley-type projections, duality  and Plancherel's theorem, we find
\eee{
\no{\mathcal{E}_1(z)}_{X^{s+a,-\frac{1}{2}+\delta}}
\lesssim
\sum_j 2^{j(s+a)}
\sup_{\no{\varphi}_{X^{0,\frac{1}{2}-\delta}}=1}S^*,
\nn
}
where
\ddd{
S^*
&=
\sum_{\substack{N_{2,1}(\xi) \\ \xi_2+\xi_4= \xi_1+\xi_3 \\ |\xi|\lesssim |\xi_{\ell}|,~\text{for some}~\ell}}
\underset{\tau_1-\tau_2+\tau_3-\tau_4+\tau_5=\tau}{\int}
|\widetilde{z}(\xi_1,\tau_1)|
|\widetilde{z}(\xi_2,\tau_2)|
|\widetilde{z}(\xi_3,\tau_3)|
|\widetilde{z}(\xi_4,\tau_4)|
|\widetilde{z_j}(\xi,\tau_5)|
|\widetilde{\varphi_j}(\xi,\tau)|
~d\sigma.
\nn
}
Without loss of generality, we assume $|\xi_1|\gtrsim |\xi|$. Then, 
\ddd{
S^*
&\lesssim
\sum_{\substack{N_{2,1}(\xi) \\ \xi_2+\xi_4= \xi_1+\xi_3 \\ |\xi|\lesssim |\xi_{\ell}|,~\text{for some}~\ell}}
\underset{\tau_1-\tau_2+\tau_3-\tau_4+\tau_5=\tau}{\int}
|\widetilde{z_{\sim j}}(\xi_1,\tau_1)|
|\widetilde{z}(\xi_2,\tau_2)|
|\widetilde{z}(\xi_3,\tau_3)|
|\widetilde{z}(\xi_4,\tau_4)|
|\widetilde{z_j}(\xi,\tau_5)|
|\widetilde{\varphi_j}(\xi,\tau)|
d\sigma
\nn\\
&\leq
\!
\sum_{\xi_1-\xi_2+\xi_3-\xi_4+\xi_5=\xi}
\underset{\tau_1-\tau_2+\tau_3-\tau_4+\tau_5=\tau}{\int}
\!\!
|\widetilde{z_{\sim j}}(\xi_1,\tau_1)|
|\widetilde{z}(\xi_2,\tau_2)|
|\widetilde{z}(\xi_3,\tau_3)|
|\widetilde{z}(\xi_4,\tau_4)|
|\widetilde{z_j}(\xi_5,\tau_5)|
|\widetilde{\varphi_j}(\xi,\tau)|
d\sigma
\nn\\
&\simeq
\int_{\mathbb{T}\times\mathbb{R}}
(z^{\#})^3 (z_{\sim j})^{\#} (z_j)^{\#}
(\varphi_j)^{\#}
~dxdt.
\nn
}
Consequently,  H\"older's inequality, Sobolev embedding  and estimate \eqref{l4-est} imply
\eee{
\no{\mathcal{E}_1(z)}_{X^{s+a,-\frac{1}{2}+\delta}}
\lesssim
\sum_j 2^{j(a-\sigma)}
\no{z}_{Y^{\frac{1}{2}+\delta,0}}^3
\no{z}_{X^{\sigma,\frac{1}{2}}}
\no{z}_{X^{s,\frac{1}{2}}},
\nn
}
where the above sum is finite for $a<\sigma=\min\{s,1\}$.
\end{proof}
\section{Nonlinear Smoothing: Proof of Theorem \ref{dbp-wp}}\label{proof1}

Having established all necessary nonlinear estimates, we now proceed to the proof of the nonlinear smoothing effect given in  Theorem \ref{dbp-wp}.
We begin by stating some useful linear estimates whose proofs  can be found in \cite{h2006}.
\begin{lemma}\label{lin-id}
Let $s\in\mathbb{R}$ and $z_0 \in H^s(\mathbb{T})$. Then, 
\ddd{\label{lin-id-est}
\no{\eta_T e^{it\partial_x^2} z_0}_{Z^s}
\lesssim
\no{z_0}_{H^s(\mathbb{T})}.
}
\end{lemma}
\begin{lemma}\label{inhomogeneous}
Let $s\in\mathbb{R}$ and $F\in Y^{s,-1}\cap X^{s,-\frac{1}{2}}$. Then,
\eee{\label{inhomogeneous-est}
\no{
\eta_T 
\int_0^t 
e^{i(t-t')\partial_x^2}
F(t')~dt'
}_{Z^s}
\lesssim
\no{F}_{Y^{s,-1}}+\no{F}_{X^{s,-\frac{1}{2}}}.
}
\end{lemma}
\begin{lemma}\label{embeddings}
Let $s\in\mathbb{R}$. Then, 
\ddd{\label{embeddings-est1}
\no{z}_{C(\mathbb{R};H^s(\mathbb{T}))}
\lesssim
\no{z}_{Z^s}.
}
Moreover, for $b_2 > b_1+\frac{1}{2}$,
\ddd{\label{embeddings-est2}
\no{z}_{Y^{s,b_1}}
\lesssim
\no{z}_{X^{s,b_2}}.
}
In particular, for all $b>\frac{1}{2}$, $X^{s,b}\hookrightarrow Z^s$.
\end{lemma}

Combining Lemmas \ref{norm-form-1}-\ref{error} with Lemmas \ref{lin-id}-\ref{embeddings}, we  deduce that any solution $z\in Z^s_T$, $s>\frac{1}{2}+\varepsilon$, $T>0$, of the Duhamel equation \eqref{duhamel-z} on $[0,T]$ with  $z_0\in H^s(\mathbb{T})$ enjoys the nonlinear smoothing effect $z-e^{it\partial_x^2}z_0\in C([0,T];H^{s+a}(\mathbb{T}))$ with the estimate \eqref{smoothing-est}. Next, we transition to the solution of the Cauchy problem \eqref{gauge-ivp}.

Let $z_0\in H^s(\mathbb{T})$ with $s>\frac{1}{2}+\varepsilon$ and take $z_0^{(n)}\in H^{\infty}(\mathbb{T})$ such that $z_0^{(n)}\rightarrow z_0$ in $H^s(\mathbb{T})$. From Theorem \ref{herr-wp}, there exist $T=T(\no{z_0}_{H^{1/2+\varepsilon}(\mathbb{T})})>0$ and functions $z\in Z^{1/2+\varepsilon}_T$ and, for $n$ sufficiently large,  $z^{(n)}\in Z^{\infty}_T$ that are solutions of the Cauchy problem \eqref{gauge-ivp} with initial data $z_0$ and $z_0^{(n)}$, respectively.
Furthermore, from the computations of Section \ref{s-dbp}, the smooth solution $z^{(n)}$ satisfies the Duhamel equation \eqref{duhamel-z} on $[0,T]$. In addition, thanks to the Lipschitz continuity of the data-to-solution map, we have $z^{(n)}\rightarrow z$ in $Z^{1/2+\varepsilon}_T$. Therefore, using Lemmas \ref{norm-form-1}-\ref{error} and the fact that $z^{(n)}$ is Cauchy in $Z^{1/2+\varepsilon}_T$, we  conclude that $z\in Z^{1/2+\varepsilon}_T$ satisfies \eqref{duhamel-z} on $[0, T]$.
In turn, as noted below Lemma \ref{embeddings}, we can employ the nonlinear smoothing estimate \eqref{smoothing-est} for an appropriate value of $a$ to infer
\ddd{
\no{z}_{Z^{\frac{1}{2}+\varepsilon+a}_T}
&\lesssim
\no{z_0}_{H^{\frac{1}{2}+\varepsilon+a}(\mathbb{T})}
+
\big\|z-e^{it\partial_x^2}z_0\big\|_{Z^{\frac{1}{2}+\varepsilon+a}_T}
\nn\\
&\lesssim
\no{z_0}_{H^{\frac{1}{2}+\varepsilon+a}(\mathbb{T})}
+
C(\no{z}_{Z^{\sigma}_T}, T) \no{z}_{Z^{\frac{1}{2}+\varepsilon}_T}
\nn\\
&\lesssim
\no{z_0}_{H^{\frac{1}{2}+\varepsilon+a}(\mathbb{T})}
+
C(\no{z_0}_{H^{\sigma}(\mathbb{T})}, T) \no{z_0}_{H^{\frac{1}{2}+\varepsilon}(\mathbb{T})}
\nn\\
&\lesssim
C(\no{z_0}_{H^{\sigma}(\mathbb{T})}, T)\no{z_0}_{H^{\frac{1}{2}+\varepsilon+a}(\mathbb{T})},
\label{z-d-est}
} 
which shows that $z\in Z^{\frac{1}{2}+\varepsilon+a}_T$ with the estimate \eqref{z-d-est}.
Iterating this process, we eventually obtain 
\eee{\nn
\no{z}_{Z^s_T}\leq C(s, \no{z_0}_{H^{\sigma}(\mathbb{T})}, T)\no{z_0}_{H^s(\mathbb{T})}.
}
This estimate shows that   $z\in Z^{1/2+\varepsilon}_T$ actually belongs to $Z^s_T$. Therefore, since $z$ satisfies the Duhamel equation \eqref{duhamel-z}, it admits the nonlinear smoothing estimate   \eqref{smoothing-est}, completing the proof of Theorem \ref{dbp-wp}.

\section{Polynomial Bound: Proof of Theorem \ref{polynomial-bound}}\label{proof2}

We shall now exploit the nonlinear smoothing effect of Theorem \ref{dbp-wp} in order to establish the polynomial bound of Theorem \ref{polynomial-bound}.
We begin by proving such a bound for the solution $z$ of the gauged Cauchy problem \eqref{gauge-ivp}. 

First, we suppose that $1\leq s \leq \frac{3}{2}-\varepsilon$ for $\varepsilon$ as in Theorem \ref{dbp-wp}. 
Fix $n\in\mathbb{N}$ and $t\in [nT, (n+1)T]$, where $T=T(\no{z_0}_{H^1(\mathbb{T})})$ is the local time of existence from Theorem \ref{herr-wp}. Then, write $z$ in the form
$$
z(t)
=
Q_{\leq n^2}z(t)
+
Q_{> n^2}z(t),
$$
where $\widehat{Q_{\leq N} z}(\xi)=\chi_{|\xi|\leq N}\widehat{z}(\xi)$ and $Q_{> N} z$ is defined similarly. 
The term $Q_{\leq n^2} z(t)$ satisfies
\eee{
\no{Q_{\leq n^2} z(t)}_{H^s(\mathbb{T})}
\leq
c\, \langle{n\rangle}^{2(s-1)}
\no{z(t)}_{H^1(\mathbb{T})}
\leq
C(\no{z_0}_{H^1(\mathbb{T})}, T) \, 
\langle{t\rangle}^{2(s-1)}.
\label{poly1}
}
The term $Q_{>n^2} z(t)$ can be handled by taking advantage of the nonlinear smoothing effect \eqref{smoothing-est}. Indeed,  
\eee{\label{n2-split}
Q_{> n^2} z(t)
=
Q_{> n^2} \big(z(t)- e^{i (t-nT)\partial_x^2} z(nT)\big)+Q_{> n^2} e^{i (t-nT)\partial_x^2} z(nT).
}
Thus,  since $t\in [nT, (n+1)T]$ and $s\leq\frac{3}{2}-\varepsilon$, we have
\ddd{
\no{Q_{> n^2} \big(z(t)- e^{i (t-nT)\partial_x^2} z(nT)\big)}_{H^s(\mathbb{T})}
&=
\Big\|J_x^{s-(\frac{3}{2}-\varepsilon)} J^{\frac{3}{2}-\varepsilon}_x Q_{> n^2} \big(z(t)- e^{i (t-nT)\partial_x^2} z(nT)\big)\Big\|_{L^2(\mathbb{T})}
\nn\\
&\leq
c\, \langle{n\rangle}^{2(s-(\frac{3}{2}-\varepsilon))} \no{Q_{> n^2} \big(z(t)- e^{i (t-nT)\partial_x^2} z(nT)\big)}_{H^{\frac{3}{2}-\varepsilon}(\mathbb{T})}
\nn\\
&\leq
\langle{n\rangle}^{2(s-(\frac{3}{2}-\varepsilon))}
C(s, \no{z(nT)}_{H^1(\mathbb{T})}, T) \no{z(nT)}_{H^1(\mathbb{T})}
\nn\\
&\leq
\langle{n\rangle}^{2(s-(\frac{3}{2}-\varepsilon))}
\tilde{C}(s, \no{z_0}_{H^1(\mathbb{T})}, T).
\label{pb-comp}
}
In order to estimate $Q_{>n^2} e^{i(t-nT)\partial_x^2} z(nT)$ in \eqref{n2-split}, we first use the strict inequality
\eee{\label{pb-comp2}
\no{Q_{>n^2} e^{i(t-nT)\partial_x^2} z(nT)}_{H^s(\mathbb{T})}
\leq
\no{Q_{>(n-1)^2} z(nT)}_{H^s(\mathbb{T})}.
}
Then, writing
$$
Q_{>(n-1)^2} z(nT)
=
Q_{>(n-1)^2}\big(z(nT)- e^{i(nT-(n-1)T)\partial_x^2} z((n-1)T)\big)+Q_{>(n-1)^2} e^{iT\partial_x^2} z((n-1)T)
$$
and proceeding similarly to \eqref{pb-comp} and \eqref{pb-comp2}, we obtain
\ddd{
\no{Q_{>(n-1)^2} z(nT)}_{H^s(\mathbb{T})}
&\leq
\langle{n-1\rangle}^{2(s-(\frac{3}{2}-\varepsilon))}
\tilde{C}(s, \no{z_0}_{H^1(\mathbb{T})}, T)
+
\no{Q_{>(n-2)^2} z((n-1)T)}_{H^s(\mathbb{T})}.
\nn
}

We may inductively continue this process to arrive at 
\ddd{
\no{Q_{> n^2} z(t)}_{H^s(\mathbb{T})}
&\leq
\sum_{k=1}^n \langle{k\rangle}^{2(s-(\frac{3}{2}-\varepsilon))} \tilde{C}(s, \no{z_0}_{H^1(\mathbb{T})}, T)+\no{z_0}_{H^s(\mathbb{T})}.
\nn
}
Then, noting that 
\eee{
\sum_{k=1}^n k^{\alpha}
\leq
c_{\alpha}
n^{\alpha+1}, \quad \alpha> -1, 
\nn
}
and observing that  $2(s-(\frac{3}{2}-\varepsilon))> -1$ since $s\geq 1$, we obtain
\ddd{
\no{Q_{> n^2} z(t)}_{H^s(\mathbb{T})}
&\leq
\sum_{k=1}^n \langle{k\rangle}^{2(s-(\frac{3}{2}-\varepsilon))} \tilde{C}(s, \no{z_0}_{H^1(\mathbb{T})}, T)+\no{z_0}_{H^s(\mathbb{T})}
\nn\\
&\leq
\langle{n\rangle}^{2(s-1+\varepsilon)} \widetilde{C}(s, \no{z_0}_{H^1(\mathbb{T})}, T)+\no{z_0}_{H^s(\mathbb{T})}
\nn\\
&\leq
\langle{t\rangle}^{2(s-1+\varepsilon)} \widetilde{\widetilde{C}}(\varepsilon, s, \no{z_0}_{H^s(\mathbb{T})}, T).
\label{poly2}
}
Combining \eqref{poly1} and \eqref{poly2} yields
\eee{\label{z-bound}
\no{z(t)}_{H^s(\mathbb{T})}
\leq
\langle{t\rangle}^{2(s-1+\varepsilon)}
C(\varepsilon, s, \no{z_0}_{H^s(\mathbb{T})}, T), \quad 1\leq s\leq \frac{3}{2}-\varepsilon.
}

Next, we consider the range $\frac{3}{2}-\varepsilon\leq s \leq 2-2\varepsilon$. We remark that the argument outlined for this range also extends to the range $s> 2-2\varepsilon$. Once again, let $t\in[nT, (n+1)T]$ and split $z(t)$ as before. For $Q_{\leq n^2} z(t)$, we still have estimate \eqref{poly1}. Also, as before, we  write
\eee{
Q_{> n^2} z(t)
=
Q_{> n^2} \big(z(t) - e^{i(t-nT)\partial_x^2} z(nT)\big) + Q_{> n^2} e^{i(t-nT)\partial_x^2} z(nT).
\nn
}
Then,
\ddd{
\no{Q_{> n^2} \big(z(t) - e^{i(t-nT)\partial_x^2} z(nT)\big)}_{H^s(\mathbb{T})}
&=
\no{J^{s-(2-2\varepsilon)}_x J^{2-2\varepsilon}_x Q_{> n^2} \big(z(t) - e^{i(t-nT)\partial_x^2} z(nT)\big)}_{L^2(\mathbb{T})}
\nn\\
&\leq
c\, \langle{n\rangle}^{2(s-(2-2\varepsilon))} 
\no{Q_{> n^2} \big(z(t) - e^{i(t-nT)\partial_x^2} z(nT)\big)}_{H^{2-2\varepsilon}(\mathbb{T})}
\nn\\
&\leq
\langle{n\rangle}^{2(s-(2-2\varepsilon))} 
C(s, \no{z(nT)}_{H^1(\mathbb{T})}, T)
\no{z(nT)}_{H^{\frac{3}{2}-\varepsilon}(\mathbb{T})}
\nn\\
&\leq
\langle{n\rangle}^{2(s-(2-2\varepsilon))} 
\tilde{C}(s, \no{z_0}_{H^1(\mathbb{T})}, T)
\no{z(nT)}_{H^{\frac{3}{2}-\varepsilon}(\mathbb{T})}.
\nn
}
In addition, the  bound \eqref{z-bound}   gives
\eee{
\no{z(nT)}_{H^{\frac{3}{2}-\varepsilon}(\mathbb{T})}
\leq
\langle{nT\rangle} \, 
C(\no{z_0}_{H^{\frac{3}{2}-\varepsilon}(\mathbb{T})}, T).
\nn
}
Consequently,
\eee{
\no{Q_{> n^2} \big(z(t) - e^{i(t-nT)\partial_x^2} z(nT)\big)}_{H^s(\mathbb{T})}
\leq
\langle{n\rangle}^{2(s-(2-2\varepsilon))+1} \tilde{C}(\varepsilon, s, \no{z_0}_{H^{\frac{3}{2}-\varepsilon}(\mathbb{T})}, T).
\nn
}
Repeating the procedure as before yields
\ddd{
\no{Q_{> n^2}z(t)}_{H^s(\mathbb{T})}
&\leq
\sum_{k=1}^n \langle{k\rangle}^{2(s-(2-2\varepsilon))+1} \tilde{C}(\varepsilon, s,\no{z_0}_{H^{\frac{3}{2}-\varepsilon}(\mathbb{T})}, T) + \no{z_0}_{H^s(\mathbb{T})}
\nn\\
&\leq
\langle{n\rangle}^{2(s-(2-2\varepsilon))+2} \tilde{C}(\varepsilon, s, \no{z_0}_{H^s(\mathbb{T})}, T) + \no{z_0}_{H^s(\mathbb{T})}
\nn\\
&\leq
\langle{t\rangle}^{2(s-(1-2\varepsilon))} \tilde{\tilde{C}}(\varepsilon, s, \no{z_0}_{H^s(\mathbb{T})}, T).
\nn
}

Overall, recalling also \eqref{z-bound}, we have established the bound
\eee{\label{z-bound2}
\no{z(t)}_{H^s(\mathbb{T})}
\leq
\langle{t\rangle}^{2(s-(1-2\varepsilon))} C(\varepsilon, s, \no{z_0}_{H^s(\mathbb{T})}, T), \quad 1\leq s\leq 2-2\varepsilon.
\nn
}
We may then repeat the above procedure to establish the bound \eqref{z-bound2} for all $s\geq 1$.

\vskip 3mm

In order to extend the result to the solution $u$ of the dNLS Cauchy problem \eqref{cauchy-dnls}, we begin by establishing the following product estimate.
\begin{proposition}\label{sobolev-product}
Let $s_1, s_2\geq s\geq 0$ and $s_1+s_2>s+\frac{1}{2}$. Then 
\ddd{\label{sobolev-product-est}
\no{fg}_{H^s(\mathbb{T})}
&\lesssim
\no{f}_{H^{s_1}(\mathbb{T})}
\no{g}_{H^{s_2}(\mathbb{T})}.
}
\end{proposition}
\begin{proof}
The argument is standard and resembles the proof of the algebra property for Sobolev spaces.
For $r, r_1, r_2\geq 0$, we claim that there exists a constant $C>0$ so that
\ddd{\label{sp-est-1}
(1+|x|+|y|)^r\leq C\left[(1+|x|)^{r+r_1} (1+|y|)^{-r_1}+(1+|y|)^{r+r_2} (1+|x|)^{-r_2}\right].
}
Indeed, first note that
\ddd{
(1+|x|+|y|)^r (1+|x|)^{r_2} (1+|y|)^{r_1}
&\leq 
(1+|x|+|y|)^{r+r_1+r_2}.
\nn
}
Thus, to establish \eqref{sp-est-1}, it suffices to show that there exists $C=C(t)>0$ so that
\ddd{\label{sp-est-2}
(1+|x|+|y|)^t
\leq 
C\left[(1+|x|)^t+(1+|y|)^t\right]
}
for all $t\geq 0$. Without loss of generality, assume that $(1+|y|)\leq (1+|x|)$. Then
\eee{
(1+|x|+|y|)^t
\leq
(1+|x|+1+|y|)^t
\leq 
2^t (1+|x|)^t
\leq 
2^t \left[ (1+|x|)^t+(1+|y|)^t\right].
\nn
}
Therefore, \eqref{sp-est-2} holds with $C=2^t$, for all $t\geq 0$. Consequently, we have \eqref{sp-est-1}. \\
\indent 
Next, we apply \eqref{sp-est-1} with $r=s$, $r_1=s_1-s$, and $r_2=s_2-s$ to see that
\ddd{
\langle{\xi\rangle}^s
\leq
(1+|\xi|)^s
\lesssim
(1+|\xi-\eta|)^{s_1} (1+|\eta|)^{s-s_1}
+
(1+|\xi-\eta|)^{s-s_2}(1+|\eta|)^{s_2}.
\nn
}
Thus,
\ddd{
\no{fg}_{H^s(\mathbb{T})}^2
&\leq
\sum_{\xi}
(1+|\xi|)^{2s}
\bigg(\sum_{\eta}
|\widehat{f}(\xi-\eta)|
|\widehat{g}(\eta)|
\bigg)^2
\nn\\
&\lesssim
\sum_{\xi}
\bigg(\sum_{\eta}
(1+|\xi-\eta|)^{s_1}
|\widehat{f}(\xi-\eta)|
(1+|\eta|)^{s-s_1}
|\widehat{g}(\eta)|\bigg)^2
\nn\\
&\hspace{0.5cm}
+
\sum_{\xi}
\bigg(\sum_{\eta}
(1+|\xi-\eta|)^{s-s_2}
|\widehat{f}(\xi-\eta)|
(1+|\eta|)^{s_2}
|\widehat{g}(\eta)|\bigg)^2,
\nn
}
where we have applied \eqref{sp-est-2} for $t=2$ in the previous inequality. It suffices to bound the first sum above, as the second sum is treated the same after a change of variables. We apply Minkowski's inequality to obtain
\ddd{
\sum_{\xi}
\bigg(\sum_{\eta}
(1+|\xi-\eta|)^{s_1}
|\widehat{f}(\xi-\eta)|
(1+|\eta|)^{s-s_1}
|\widehat{g}(\eta)|\bigg)^2
&\leq
\no{f}_{H^{s_1}(\mathbb{T})}^2
\bigg(
\sum_{\eta} (1+|\eta|)^{s-s_1} |\widehat{g}(\eta)|
\bigg)^2.
\nn
}
Finally, due to the Cauchy-Schwarz inequality,
\ddd{
\sum_{\eta} (1+|\eta|)^{s-s_1} |\widehat{g}(\eta)|
&=
\sum_{\eta} (1+|\eta|)^{s-s_1-s_2} (1+|\eta|)^{s_2} |\widehat{g}(\eta)|
\nn\\
&\leq
\no{g}_{H^{s_2}(\mathbb{T})} 
\no{(1+|\eta|)^{-(s_1+s_2-s)}}_{\ell^2_{\eta}},
\nn
}
where the above norm is finite due to the fact that $s_1+s_2-s>\frac{1}{2}$.
\end{proof}
\begin{lemma}\label{u-z}
Let $s\geq 0$. Then, there exists $C>0$ such that
\ddd{\label{u-z-bound}
\no{u(t)}_{H^s(\mathbb{T})}
\leq
C\Big(1+\no{z(t)}^2_{H^{\frac{1}{4}}(\mathbb{T})}\Big)\no{z(t)}_{H^s(\mathbb{T})}.
}
\end{lemma}
\begin{proof}
Clearly, $\no{z(t)}_{H^s(\mathbb{T})}=\no{w(t)}_{H^s(\mathbb{T})}$, for $w$ given by \eqref{w-gauge} . Furthermore, since $\widehat{w}(\xi,t)=e^{-2i\xi\mu t} \widehat{v}(\xi,t)$, for $v$ given by \eqref{v-gauge}, we have that $\no{w(t)}_{H^s(\mathbb{T})}=\no{v(t)}_{H^s(\mathbb{T})}$. Thus, it suffices to establish \eqref{u-z-bound} for $z(t)$ replaced by $v(t)$. \\ 
\indent 
Recall that $u(x,t)=e^{i\mathcal{I}(v)(x,t)}v(x,t)$. For $\delta>0$, we apply Proposition \ref{sobolev-product} with $s_1=s$ and $s_2=\frac{1}{2}+\delta$ to obtain
\eee{
\no{u(t)}_{H^s(\mathbb{T})}
=
\no{e^{i\mathcal{I}(v)(t)}v(t)}_{H^s(\mathbb{T})}
\leq
C \no{e^{i\mathcal{I}(v)(t)}}_{H^{\frac{1}{2}+\delta}(\mathbb{T})} \no{v(t)}_{H^s(\mathbb{T})}.
\nn
}
Now, observe that
\ddd{
\no{e^{i\mathcal{I}(v)(t)}}_{H^{\frac{1}{2}+\delta}(\mathbb{T})}
&\leq
\no{e^{i\mathcal{I}(v)(t)}}_{H^1(\mathbb{T})}
=
\no{e^{i\mathcal{I}(v)(t)}}_{L^2(\mathbb{T})}
+
\no{\partial_x \mathcal{I}(v)(t)}_{L^2(\mathbb{T})}
\nn\\
&=
\sqrt{2\pi}+\no{ |v(t)|^2 -\mu}_{L^2(\mathbb{T})}
\leq
C\left(1+\no{v(t)}_{L^4(\mathbb{T})}^2\right)
\leq 
C\Big(1+\no{v(t)}^2_{H^{\frac{1}{4}}(\mathbb{T})}\Big),
\nn
}
where the last inequality follows from Sobolev embedding.
\end{proof}
\begin{remark}
It is clear from the proof of Lemma \ref{u-z} that the estimate \eqref{u-z-bound} holds with the roles of $u$ and $z$ switched.
\end{remark}
Returning to the case that $s\geq 1$ and applying Lemma \ref{u-z} and the bound \eqref{z-bound} yields
\ddd{
\no{u(t)}_{H^s(\mathbb{T})}
&\leq
C\Big(1+\no{z(t)}_{H^{\frac{1}{4}}(\mathbb{T})}^2\Big)\no{z(t)}_{H^s(\mathbb{T})}
\nn\\
&\leq
C(\no{z_0}_{H^1(\mathbb{T})})\no{z(t)}_{H^s(\mathbb{T})}
\nn\\
&\leq 
C(\no{z_0}_{H^1(\mathbb{T})}) \, \tilde{C}(\varepsilon, s,\no{z_0}_{H^s(\mathbb{T})}, T) \, \langle{t\rangle}^{2(s-1)+\varepsilon}
\nn\\
&\leq
C(\varepsilon,s,\no{u_0}_{H^s(\mathbb{T})}, T) \, \langle{t\rangle}^{2(s-1)+\varepsilon},
\nn
}
which concludes the proof of Theorem \ref{polynomial-bound}.
%
%
%%%%%%%%%%%%%%

\end{document}